\documentclass{article}

\usepackage{algorithmic}
\usepackage{algorithm}
\usepackage{amsmath}
\usepackage{amsthm}
\usepackage{amssymb}
\usepackage{times, mathptmx}

\usepackage[latin1]{inputenc}

\newcommand{\MORPH}{\mathcal{S}}
\newcommand{\QNUM}{D}

\newcommand{\HYBINP}{\mathcal{U}}

\newcommand{\COL}{\mathbf{C}}

\newcommand{\LSS}{DTLSS}
\newcommand{\BLSS}{DTLSS\ }
\newcommand{\SLSS}{DTLSSs}
\newcommand{\BSLSS}{DTLSSs\ }
\newcommand{\GCR}{\textbf{GCR}}

\newcommand{\RFSF}{FFS\ }
\newcommand{\NFSF}{FFS}

\newcommand{\POWO}{d}

\newcommand{\WR}{span-reachable}

\newcommand{\SPAN}{\mathrm{Span}}

\newcommand{\SwitchSysLin}[1][{}]{ (p,m,n^{#1},Q,\{(A_{q}^{#1},B_{q}^{#1},C_{q}^{#1}) \mid q \in Q \},x_0^{#1})
}

\newcommand{\IM}{\mathrm{Im}}
\newcommand{\Rank}{\mathrm{rank}\mbox{ } }
\newtheorem{Theorem}{Theorem}
\newtheorem{Lemma}{Lemma}
\newtheorem{Definition}{Definition}
\newtheorem{Remark}{Remark}
\newtheorem{Notation}{Notation}
\newtheorem{Example}{Example}

\newtheorem{Procedure}{Procedure}

\renewcommand{\paragraph}[1]{\smallskip\noindent\textbf{#1.} }
\renewcommand{\Re}{\mathbb{R}}

\title{ Realization theory of discrete-time linear switched systems}
\date{}
\author{Mih\'aly Petreczky 
 \footnote{
   Univ Lille Nord de France, F-59000 Lille, France, and EMDouai, IA, F-59500 Douai, France,
      \texttt{mihaly.petreczky@mines-douai.fr.}}
  Laurent Bako \footnote{
   Univ Lille Nord de France, F-59000 Lille, France, and EMDouai, IA, F-59500 Douai, France,
      \texttt{laurent.bako@mines-douai.fr.}}
   \ and Jan H. van Schuppen
   \footnote{
   Centrum Wiskunde en Informatica (CWI)
    P.O.Box 94079, 1090GB Amsterdam, The Netherlands
   \texttt{J.H.van.Schuppen@cwi.nl}}
}

\begin{document}
\maketitle

\begin{abstract}
 The paper presents realization theory of discrete-time
 linear switched systems (abbreviated by \SLSS).
 We present necessary and sufficient
 conditions for an input-output map to admit a
 discrete-time linear switched state-space realization.
 In addition, we
 present a characterization of minimality of
 discrete-time linear switched systems in terms of
 reachability and observability.
 Further, we prove that minimal realizations are unique up to
 isomorphism. We also discuss algorithms for converting a
 linear switched system to a minimal one and for constructing
 a state-space representation from input-output data.
 The paper uses the theory of rational formal power series
 in non-commutative variables. 

 \medskip

\end{abstract}

\section{Introduction}
 In this paper we develop realization theory of
 discrete-time linear switched systems (abbreviated by \SLSS).
  \BSLSS are one of the simplest and best studied 
  classes of hybrid systems, \cite{Sun:Book}. A \LSS\ is a
  discrete-time switched system, such that the continuous
  sub-system associated with each discrete state is linear.
  The switching signal is viewed as an external input, and all
  linear systems live on the same input-output- and state-space.

 \textbf{Realization theory.}
  \emph{Realization theory} is one of the central topics of
  system theory. For \SLSS,  the subject of realization theory is to
  answer the following questions.
  \begin{itemize}
  \item
   When is it possible to construct
   a (preferably minimal) \LSS\ state-space representation
  of the specified input/output behavior ?
 \item
   How to characterize  minimal \BSLSS  which generate
   the specified input/output behavior ?
 \end{itemize}

 \textbf{Motivation.}
 While there is a substantial literature on
 linear switched systems, realization theory was addressed only for the
 continuous-time case \cite{MP:BigArticlePartI,MP:RealForm}. 
The motivation for devoting a separate paper to realization theory
of discrete-time \BSLSS is the following.
\begin{enumerate}
\item
\label{Diff1}
 Realization theory for \BSLSS\ is substantially different
 from realization theory for linear systems.
\item
\label{Diff2}
 Realization theory for \BSLSS\ 
 is substantially different from
 the continuous-time case. More precisely,
 the realization problem both for continuous-time linear switched systems
 and for \BSLSS can be transformed to the same realization problem for
 formal power series. The difference lies in the specific transformation.

\item
\label{Diff3}
  Formulating realization theory explicitly for discrete-time \BSLSS
  will be useful the identification of these systems. In fact,
  the results of this paper were already used in \cite{MP:HSCC2010}
  for analyzing identifiability of \BSLSS.
\end{enumerate}
Intuitively, the main difference between linear realization theory and
that of linear switched systems is the following.
For linear switched systems, the realization problem is equivalent 
to the problem of
representing a sequence of numbers (Markov-parameters) 
as products of several non-commuting matrices 
(pre- and post-multiplied by fixed matrices).
For linear case, the corresponding problem involves not products of
non-commuting matrices, but powers of one matrix.
In addition, for linear switched systems
 we allow arbitrary non-zero initial state. 
 The presence of
a non-zero initial state means that the input response and initial-state response have to be decoupled. A similar approach was already described
in \cite{Hof} for linear systems.


 \textbf{Contribution of the paper}
   We prove that span-reachability and observability of
   \BSLSS is equivalent to minimality and that minimal realizations
   are isomorphic. We also show that any \BLSS can be transformed
   to a minimal one while preserving its input-output behavior,
   by presenting a minimization algorithm.
   In addition, we formulate the concept of
   Markov-parameters and Hankel-matrix for \BSLSS. We show that
   an input-output map can be realized by a \LSS\ if and only if
   the Hankel-matrix is of finite rank. We also present a procedure
   for constructing a \LSS\ state-space representation from the
   Hankel-matrix.
	Our main tool is the theory of rational formal power series \cite{Reut:Book,Son:Real}. 

 \textbf{Related work}
  To the best of our knowledge, the results of this paper are new.
  The results on minimality of \BSLSS were already announced
  in \cite{MP:HSCC2010}, but no detailed proof was provided.
  The results on existence of a realization by a \LSS\ were not
  previously published.

  The realization problem for hybrid systems was first formulated in
  \cite{GrossmanHybAlg}.
  In \cite{Paoletti2,Weiland06} the relationship between input-output equations and the state-space representations was studied.
  In 
  \cite{MP:Phd,MPRV:HSCC08,Petreczky09-TAC} realization theory
  for various classes of hybrid systems were developed.
 In particular, realization theory for continuous-time (bi)linear switched
 systems was developed in \cite{MP:BigArticlePartI,MP:RealForm}.
 The approach of the present paper is similar to that of
 \cite{MP:BigArticlePartI}, however the details of the
 steps are different. 
 There is a vast literature on topics related to realization theory, such as
 system identification, observability and reachability of hybrid systems, see
 \cite{IdentHybTut,SchuColObs,Sun:Book,Bako09-SYSID2,Bako08-IJC,VidalAutomatica,JLSVidal,Verdult04,Juloski05-TAC,Roll04,Paoletti1,Ferrari03,Nakada05,Weiland06,Paoletti2}.

 Our main tool for developing realization theory of
 \BSLSS is the theory of rational formal power series.
 This theory was already used for
 realization theory of nonlinear and multi-dimensional systems,
 \cite{MFliessHank,isi:tac,Son:Real,ball:1474}. 
 State-affine systems from \cite{Son:Real} include
 autonomous \BSLSS as a special case. Realization theory
 of state-affine systems is equivalent to that of rational formal power series.
 In this paper we reduce the realization problem for \BSLSS directly to that of
 rational formal power series. Hence, indirectly we also show 
 that the realization problems for
 \BSLSS and state-affine systems are equivalent. 
 One could probably reduce the realization problem for \BSLSS to that of
 state-affine systems directly, however it is unclear if such a reduction
 would be more advantageous. 

\textbf{Outline}
  \S \ref{sect:lin} presents a brief overview of realization 
  theory of discrete-time linear systems.
  \S \ref{sect:switch} presents the formal definition of \BSLSS and it
  formulates the major system-theoretic concepts for this system class.
  \S \ref{sect:main_results:min} -- \S \ref{sect:main_results:ex}
  states the main results of the paper.
  \S \ref{sect:pow} contains the necessary 
  background on the theory of rational formal
  power series. 
  The proofs are presented in \S \ref{sect:proof} and Appendix \ref{appA}.
 
\textbf{Notation}
 Denote by $\mathbb{N}$ the set of natural numbers including $0$.
 The notation described below is standard in automata theory, see \cite{GecsPeak,AutoEilen}.
 Consider a set $X$ which will be called the \emph{alphabet}.
 Denote by $X^{*}$ the set of finite
 sequences of elements of $X$.  
 Finite sequences of elements of
 $X$ are be referred to as \emph{strings} or \emph{words} over $X$.
 Each non-empty word $w$ is of the form $w=a_{1}a_{2} \cdots a_{k}$
 for some $a_1,a_2,\ldots,a_k \in X$.
 The element $a_i$ is called the \emph{$i$th letter of $w$}, for
 $i=1,\ldots,k$ and $k$ is called the \emph{length $w$}.
 We denote by $\epsilon$ the \emph{empty sequence (word)}.
 The length of word $w$ is denoted by $|w|$;note that $|\epsilon|=0$.
 We denote by $X^{+}$ the set
 of non-empty words, i.e. 
 $X^{+}=X^{*}\setminus \{\epsilon\}$.
 We denote by $wv$ the concatenation of word $w \in X^{*}$ with $v \in X^{*}$.
  We use the notation of \cite{JacobAlg1} for matrices indexed by sets
  other than natural numbers. 
For each $j=1,\ldots,m$, $e_j$ is the $j$th unit vector of $\mathbb{R}^{m}$, i.e.
  $e_j=(\delta_{1,j},\ldots, \delta_{n,j})$,
  $\delta_{i,j}$ is the Kronecker symbol. 

 \section{Realization theory for linear systems}
\label{sect:lin}
  In this section we present a brief review of
  realization theory of discrete-time linear systems, based on \cite{Hof}.
  Although the results of this section are not used in the paper,
  they help to get an intuition for
  the results on realization theory of \BSLSS.

  The input-output maps of interest are 
  of the form \( y:(\mathbb{R}^{m})^{+} \rightarrow \mathbb{R}^p\).
  For each sequence $u=u_0\cdots u_t$, $t \ge 0$,
  $y(u)$ is the output of the underlying system
  at time $t$, if inputs $u_0,\ldots,u_t$ are fed.
  It is well-known that for $y$ to be  realizable by a 
  linear system, it must be of the form
  \begin{equation}
  \label{eq:LinearSys-Markov} 
   y(u_0\cdots u_t) = K_t+\sum_{j=0}^{t-1} H_{t-j-1}u_{j} 
  \end{equation}
  for some matrices $K_k \in \mathbb{R}^p$, $H_k \in \mathbb{R}^{p \times m}$, $k=0,1,2,\ldots,$ and for any sequence of inputs
   $u_0,\ldots,u_t \in \mathbb{R}^m$.
 Consider a discrete-time linear system
  \begin{equation}
  \label{rev2}
   \Sigma\left\{
     \begin{split}
      & x_{t+1} = Ax_t+Bu_t \mbox{ where } x_0 \mbox{ is fixed } \\
      & y_t =Cx_t
     \end{split}\right.
  \end{equation}
  where $A$, $B$ and $C$ are $n \times n$, $n \times m$ and
  $p \times n$ real matrices and $x_0 \in \mathbb{R}^n$
  is the initial state.
  Note that the initial
  state is $x_0$, and $x_0$ need not be zero.
  The map
  $y$ is said to be \emph{realized by} $\Sigma$,
  if the output response of
  $\Sigma$ to any input $u$ equals $y(u)$. This is the case
  if and only if $y$ is of the form \eqref{eq:LinearSys-Markov}, and
  $K_t=CA^tx_0$, $H_t=CA^tB$, $t \ge 0$. We call $\Sigma$ a \emph{minimal
  realization of $y$}, if it has the smallest state-space dimension among
  all the linear system realizations of $y$.
  \begin{Theorem}[\cite{Hof}]
  \label{lin:theo:min}
   Assume that $\Sigma$ is a linear system realization of $y$.
   Then $\Sigma$ is a minimal realization of $y$, if and only
   if it is weak-reachable and observable.
   Recall that $\Sigma$ is weak-reachable if and only if
 $(A,\begin{bmatrix} x_0 & B \end{bmatrix})$ is a reachable pair.
   All minimal realizations of $y$ are isomorphic and any realization
   of $y$ can be transformed to a minimal one.
  \end{Theorem}
  The transformation to a minimal system can be carried out
  by first transforming the linear system to a weak-reachable one,
  and then to an observable one, \cite{Hof}.

  Next, we formulate conditions for existence of a linear system realization 
  of $y$. 
  To this end, we assume that $y$ is of the form \eqref{eq:LinearSys-Markov}.
  This assumption is necessary (but not sufficient) for existence of 
  a realization. We call the matrices  $M_t=\begin{bmatrix} K_t & H_t \end{bmatrix}$, $t \ge 0$ \emph{Markov parameters}. This terminology is slightly
  different from the one used in \cite{Hof}.
  Note that $y$  is completely determined by the 
  Markov-parameters $\{M_{t}\}_{t=0}^{\infty}$.
  In addition, note that
  we defined the Markov-parameters without assuming
  the existence of a linear system realization.
  In fact, we use the Markov-parameters  
  for characterizing the existence of a linear system realization. 
  More precisely, 
  we define the infinite block \emph{Hankel-matrix $H_y$ of $y$} as
  follows $H_y=(H_{i,j})_{i,j=1}^{\infty}$, $H_{i,j}=M_{i+j-2}$,  i.e. the entries of
  $H_y$ are formed by the entries of the Markov-parameters of $y$.
  \begin{Theorem}[\cite{Hof}]
  \label{lin:theo:ex}
   The map
   $y$ can be realized by a linear system if and only if
  the rank of $H_y$ is finite. If $\Rank H_y =n < +\infty$, 
  then a minimal linear system realization $\Sigma$ of $y$ 
  can be constructed from the columns of
  $H_y$.
  In particular, this means that
  $\Rank H_y$ equals the dimension of any minimal linear system which
  is a  realization of $y$.  
  \end{Theorem}
  \begin{Procedure}
  \label{lin:hank}
 The construction of $\Sigma$ from the columns of $H_y$ is as follows.
  Fix a finite basis in the column space of $H_y$. Then $x_0$ is
  formed by the coordinates of the first column of $H_y$ in this basis,
  the $r$th column of the matrix $B$ represents the coordinates of the $r+1$th
  column of $H_y$ in this basis. The matrix $C$ is the matrix 
  (in the fixed basis)
  of the linear map which maps each column to the vector formed by
  its first $p$ entries. Finally, $A$ is the matrix (in the fixed 
  basis) of the linear map which maps the $j$th column to the $j+(m+1)$th column, i.e.
  it maps the block column $(M_{i+j-2})_{i=1}^{\infty}$ to the
  block column $(M_{i+j-1})_{i=1}^{\infty}$.
 \end{Procedure}

\section{Linear switched systems}
\label{sect:switch}
 In this section we present the formal definition of \BSLSS along with
 a number of relevant system-theoretic concepts for \BSLSS.
 \begin{Definition}
  \label{switch:def}
  Recall from \cite{MP:HSCC2010} that a discrete-time
  linear switched system (abbreviated by \LSS), is a discrete-time
  control system of the form
  \begin{equation}
  \label{lin_switch0}
  \Sigma\left\{
  \begin{array}{lcl}
   x_{t+1} &=& A_{q_t}x_t+B_{q_t}u_t \mbox{ and $x_0$ is fixed}  \\
   y_t  &=& C_{q_t}x_t. 
  \end{array}\right.
  \end{equation}
  Here $Q=\{1,\ldots,\QNUM\}$ is the finite set of discrete modes, 
  $\QNUM$ is a positive integer. For each $t \in \mathbb{N}$,
 $q_t \in Q$ is the
  discrete mode, $u_t \in \mathbb{R}$ is the continuous input,
  $y_t \in \mathbb{R}^{p}$ is the output at time $t$. Moreover, 
  $A_{q} \in \mathbb{R}^{n \times n}$,
  $B_{q} \in \mathbb{R}^{n \times m}$, $C_q \in \mathbb{R}^{p \times n}$
  are the matrices of the linear system in mode $q \in Q$,  and $x_0$
  is the initial continuous state. 
    We will use
    \[ \SwitchSysLin \] as a
short-hand notation for \SLSS\  of the form (\ref{lin_switch0}).
\end{Definition}
 Throughout the section, \emph{$\Sigma$ denotes a \LSS\ of the form (\ref{lin_switch0}).}
The \emph{inputs of $\Sigma$ are the
continuous inputs $\{u_t\}_{t=0}^{\infty}$ and the
switching signal $\{q_t\}_{t=0}^{\infty}$.} The state of the system
at time $t$ is $x_t$.
Note \emph{that any switching
signal is admissible.}
We use the following notation for the inputs of $\Sigma$.
\begin{Notation}[Hybrid inputs]
 Denote $\HYBINP=Q \times \mathbb{R}^{m}$.
\end{Notation}
 We denote by  $\HYBINP^{*}$ (resp. $\HYBINP^{+}$) 
 the set of all 
  finite 
 (resp. non-empty and finite)
sequences
 of elements of $\HYBINP$.
 A sequence 
 \begin{equation}
 \label{inp_seq}
 w=(q_0,u_0)\cdots (q_t,u_t) \in \HYBINP^{+} \mbox{, } t \ge 0 
 \end{equation}
  describes the scenario, when the discrete mode $q_i$ and
  the continuous input $u_i$ are fed to $\Sigma$ at
 time $i$, for $i=0,\ldots,t$.
    \begin{Definition}[State and output]
     Consider a state $x_{init} \in \mathbb{R}^{n}$.
     For any $w \in \HYBINP^{+}$ of the form (\ref{inp_seq}),
     denote by \( x_{\Sigma}(x_{init},w) \)
     \emph{the state of $\Sigma$}
     at time $t+1$, and denote by
     $y_{\Sigma}(x_{init},w)$ the \emph{output} of $\Sigma$
     at time $t$, if $\Sigma$ is started from $x_{init}$ and 
     the inputs $\{u_i\}_{i=0}^t$ and
     the discrete modes $\{q_i\}_{i=0}^{t}$ are fed to the system. 
   For notational purposes, 
     we define 
     $x_{\Sigma}(x_{init},\epsilon)=x_{init}$.
  \end{Definition}    
     That is,
     $x_{\Sigma}(x_{init},w)$ is defined recursively as follows;
     $x_{\Sigma}(x_{init},\epsilon)=x_{init}$,  and if $w=v(q,u)$
     for some $(q,u) \in \HYBINP$, $v \in \HYBINP^{*}$, then
     \[ x_{\Sigma}(x_{init},w)=A_{q}x_{\Sigma}(x_{init},v)+B_{q}u. \]
     If $w \in \HYBINP^{+}$ and
      $w=v(q,u)$, $(q,u) \in \HYBINP$, $v \in \HYBINP^{*}$,
      then 
      \[ y_{\Sigma}(x_{init},w)=C_{q}x_{\Sigma}(x_{init},v). \]
  \begin{Definition}[Input-output map]
   The map  $y_{\Sigma}: \HYBINP^{+} \rightarrow \mathbb{R}^{p}$,
   defined by
   $\forall w \in \HYBINP^{+}: y_{\Sigma}(w)=y(x_0,w)$, is called
   the input-output map of $\Sigma$.
  \end{Definition}   
   That is, the input-output map of $\Sigma$ maps each
   sequence $w \in \HYBINP^{+}$ to the output generated
   by $\Sigma$ under the hybrid input $w$, if started from the initial state $x_0$.
   The definition above implies that the input-output behavior
   of a \BLSS can be formalized as a map
   \begin{equation}
   \label{io_map}
     f:\HYBINP^{+} \rightarrow \mathbb{R}^{p}. 
   \end{equation}
   The value $f(w)$ for $w$ of the form \eqref{inp_seq} represents the output 
   of the underlying black-box system at time $t$, 
   if the continuous inputs $\{u_i\}_{i=0}^{t}$ and the switching sequence 
   $\{q_i\}_{i=0}^t$ are fed to the system.
   This black-box system may or may not admit a description by a \LSS.

   Next, we define
   when a general map $f$ of the form \eqref{io_map} 
   is adequately described by the
   \BLSS $\Sigma$, i.e. when $\Sigma$ is a realization of $f$.
   \begin{Definition}[Realization]
   \label{switch_sys:real:def1}
    The \BLSS $\Sigma$ is a \emph{realization} of
    an input-output map $f$ of the form \eqref{io_map}, if
   $f$ equals the input-output map of $\Sigma$, i.e. $f=y_{\Sigma}$.
   \end{Definition}
   The \emph{reachable set} $Reach(\Sigma)$ of $\Sigma$ 
   is  the set of all states
   which can be reached from the initial state $x_0$ of $\Sigma$, i.e.
  \begin{equation*}
  \label{sect:problem_form:reach0}
    Reach(\Sigma) = \{ x_{\Sigma}(x_{0},w) \in \mathbb{R}^{n}
   \mid  w \in \HYBINP^{*} \} 
   \end{equation*}
  \begin{Definition}[(Span-)Reachability)]
    The \LSS\  $\Sigma$ is \emph{reachable}, if
    $Reach(\Sigma)=\mathbb{R}^n$, and $\Sigma$ is \emph{\WR} if
    $\mathbb{R}^{n}$ is the smallest vector space containing 
    $Reach(\Sigma)$.
  \end{Definition}
   Reachability implies span-reachability but in general
they are not equivalent. 
   \begin{Definition}[Observability]
    The \LSS\  $\Sigma$ is called \emph{observable} if
    for any two states $x_1,x_2 \in \mathbb{R}^n$ of $\Sigma$, 
    \[ 
      (\forall w \in \HYBINP^{+}: y_{\Sigma}(x_1,w)=y_{\Sigma}(x_2,w)) \implies x_1=x_2 \]
   \end{Definition}
   That is, observability means that if we pick any two states of the system, then for \textbf{some} continuous input and switching signal, the resulting
  outputs will be different.
   \begin{Definition}[Dimension]
   \label{switch_sys:dim:def}
    The dimension of $\Sigma$, denoted by $\dim \Sigma$, is 
    the dimension $n$ of its state-space.
   \end{Definition}
 Note that the number of discrete states is fixed, and
 hence it is not included into the definition of dimension.
The reason for this is the following. 
We are interested in realizations of input-output maps, which map continuous inputs and switching signals to continuous outputs. Hence, for all possible \BLSS realizations, the set of discrete modes is fixed.
\begin{Definition}[Minimality]
 Let 
 $f$ be an
 input-output map.
 Then $\Sigma$ is \emph{a minimal realization of $f$}, if
 $\Sigma$ is a realization of $f$, and
 for any \LSS\ 
 $\hat{\Sigma}$ which is a realization 
 of $f$, $\dim \Sigma \le \dim \hat{\Sigma}$.
\end{Definition}

  \begin{Definition}[\BLSS morphism]
  \label{sect:problem_form:lin:morphism}
   Consider a \LSS\  $\Sigma_1$ of the form (\ref{lin_switch0}) and
   a \LSS\  $\Sigma_2$ of the form 
   \[ \Sigma_{2}=\SwitchSysLin[a] \]
    Note that $\Sigma_1$ and $\Sigma_2$ have the same set of
    discrete modes.
    A matrix
    $\MORPH \in \mathbb{R}^{n^a \times n}$
    is said to be a \emph{\LSS\  morphism}
    from $\Sigma_{1}$ to $\Sigma_{2}$,  denoted by
    $\MORPH:\Sigma_{1} \rightarrow \Sigma_{2}$, if 
   \begin{equation*}
   \begin{array}{ll}
    \MORPH x_0=x_0^a, \mbox{ and } &
    \forall q \in Q: 
    A^{a}_{q}\MORPH=\MORPH A_{q}\mbox{,\ \ }  B_{q}^{a}=\MORPH B_{q}
    \mbox{,\ \ }
    C_{q}^{a}\MORPH =C_{q}.
   \end{array}
  \end{equation*}
  The morphism $\MORPH$ is called surjective ( injective ) if
  $\MORPH$ is surjective ( injective ) as a linear map. The morphism
  $\MORPH$ is said to be a \LSS\  isomorphism,
  if it is an isomorphism as a linear map.
  \end{Definition}
 

\section{Main result on minimality}
\label{sect:main_results:min}
 Below we present the main results of the paper on
minimality of \SLSS. In addition, we present a minimization
procedure and rank tests for checking minimality.
In the sequel, \emph{$\Sigma$ denotes a \LSS\ of the form
(\ref{lin_switch0}), and $f$ denotes an input-output map
$f:\HYBINP^{+} \rightarrow \mathbb{R}^{p}$.}
\begin{Theorem}[Minimality]
\label{theo:min}
 \begin{enumerate}
 \item
 A \BLSS realization of $f$ is minimal, if and only if
 it is \WR\ and observable. 
 \item
 All minimal \BLSS realizations of $f$
 are isomorphic. 
 \item
 Every \BLSS realization of $f$
 can be converted to a minimal \BLSS realization of $f$
 (see Procedure \ref{LSSmin} below).
 \end{enumerate}
\end{Theorem}
The proof of Theorem \ref{theo:min} is presented in \S \ref{sect:proof}.
\begin{Remark}
 Note that $\Sigma$ can be minimal, while none of the linear subsystems
 is minimal, see Example \ref{example1} below.
 Since all minimal realizations are isomorphic, it then follows
 that such a \BLSS cannot be transformed to a one where at least one
 subsystem is minimal without loosing input-output behavior.
\end{Remark}
For analogous theorem for continuous-time linear switched systems
see \cite{MP:BigArticlePartI,MP:RealForm}.
Intuitively, the theorem says the following. First, a minimal
\BLSS should not contain states which are not linear
combination of the reachable ones (hence span-reachability).
Second, a minimal \BLSS\ should not contain multiple
states which 
exhibit the same input-output behavior (hence observability).
 Next, we present rank conditions for observability and span-reachability.
 These conditions can be used to test minimality and to formulate
  Procedure \ref{LSSmin}.
 \begin{Notation}
 \label{repr:not1}
   Let $X$ be a finite set, $\mathcal{X}$ be a linear space, 
$A_{\sigma}: \mathcal{X} \rightarrow \mathcal{X},\sigma \in X$ 
   be linear maps and let $w \in X^{*}$.  The linear map $A_w$ on 
   $\mathcal{X}$ is defined as follows.
   If $w=\epsilon$, then $A_{\epsilon}$ is the identity map, i.e
   $A_{\epsilon}x=x$ for all $x \in \mathcal{X}$.
   If  $w=\sigma_1\sigma_2 \cdots \sigma_{k} \in X^{*}$, 
   $\sigma_1, \cdots \sigma_k \in X$, $k > 0$, 
  then 
  \begin{equation} 
  \label{repr:eq2}
  A_{w}=A_{\sigma_{k}}A_{\sigma_{k-1}} \cdots A_{\sigma_{1}}. 
  \end{equation}
   If $\mathcal{X}=\mathbb{R}^n$ for some $n > 0$, then $A_w$ and
   each $A_{\sigma}$, $\sigma \in X$
  can be identified with an $n \times n$ matrix. In this case
  $A_w$ defines a product of matrices.
\end{Notation}
 We denote by $Q^{<n}$ the set $\{w \in Q^{*} \mid |w| < n\}$
of all words $w \in Q^{*}$ of length at most $n-1$.
We denote by $M_n$ the cardinality of $Q^{<n}$ and we fix an
enumeration
 \[ Q^{<n} = \{v_1,\ldots,v_{M_n}\}. \]
We will use the notation defined above to define observability and
reachability matrices for \BSLSS.
\begin{Theorem}
\label{sect:problem_form:reachobs:prop1}
\label{sect:real:lemma1}
 \textbf{Span-Reachability.} 
    Define the \emph{span-reachability matrix} 
    $R(\Sigma)$ of $\Sigma$
    \begin{equation*}
    \begin{split}
       & \mathcal{R}(\Sigma)=\begin{bmatrix}
          A_{v_1}\widetilde{B}, & A_{v_2}\widetilde{B}, & \cdots, &  A_{v_{M_{n}}}\widetilde{B}
        \end{bmatrix} \in \mathbb{R}^{n \times (|Q|m+1)M_n} \mbox{ where } \\
         & \widetilde{B}=\begin{bmatrix} x_0, & B_1, & \cdots, & B_{\QNUM} 
       \end{bmatrix}
    \end{split}
    \end{equation*}
   Then $\Sigma$ is \WR\  if and only if
   $\Rank \mathcal{R}(\Sigma)=n$.

 \textbf{Observability.}
    Define the \emph{observability matrix} $O(\Sigma) \in \mathbb{R}^{p|Q|M_n \times n}$ of $\Sigma$ as
    follows.
    \[  O(\Sigma)=\begin{bmatrix} \widetilde{C}A_{v_1} \\ \widetilde{C}A_{v_2} \\ \vdots \\ \widetilde{C}A_{v_{M_n}} \end{bmatrix} 
  \mbox{ where } \widetilde{C}=\begin{bmatrix} C_1 \\ C_2 \\ \vdots \\ C_{\QNUM} \end{bmatrix} 
    \]
    Then $\Sigma$ is observable if and only if 
    \( \Rank O(\Sigma) = n \).
\end{Theorem}
 Informally, $\mathcal{R}(\Sigma)$ is formed by horizontal
 concatenation of blocks $A_{w}B_q$, for all  $w \in Q^{<n}$, $q \in Q$, and
 $O(\Sigma)$ is the vertical concatenation of blocks
 $C_{q}A_w$, $q \in Q$, $w \in Q^{<n}$.
 Notice that if $Q=\{1\}$, then $\mathcal{R}(\Sigma)$ is the
controllability matrix of $(A_1, \begin{bmatrix} x_0 & B_1 \end{bmatrix})$
 and $O(\Sigma)$ is the observability matrix of $(C_1,A_1)$.
 Hence,
 the linear system $(A_1,B_1,C_1,x_0)$ is weak-reachable (observable)
 if and only if it is span-reachable (observable), if interpreted
 as a \LSS. Hence, Theorem \ref{theo:min} implies
 Theorem \ref{lin:theo:min}. 

 The result of Theorem \ref{sect:real:lemma1} follow
 from \cite{Sun:Book}, the detailed proof can be found in
  Appendix \ref{appA}.
Next, we formulate
procedures for reachability, observability and minimality reduction of \BSLSS. 
\begin{Procedure}[Reachability reduction]
 \label{LSSreach}
 Assume $\dim \mathcal{R}(\Sigma)=n^r$ and choose a basis $b_1,\ldots,b_n$ of 
 $\mathbb{R}^n$ such that $b_1,\ldots,b_{n^{r}}$ span $\IM \mathcal{R}(\Sigma)$.
 In the new basis, $A_q,B_q,C_q$, $q \in Q$ and $x_0$ become as follows
 \[ A_{q}=\begin{bmatrix} A_{q}^r, & A^{'}_{q} \\ 0, & A^{''}_{q} \end{bmatrix},
    C_{q}=\begin{bmatrix} C_q^r, & C_{q}^{nr} \end{bmatrix},
    B_{q}=\begin{bmatrix} B_{q}^r \\ 0 \end{bmatrix},
    x_0=\begin{bmatrix} x_0^r \\ 0 \end{bmatrix}
 \]
where $A^r_{q} \in \mathbb{R}^{n^r \times n^r}, B_q^r \in \mathbb{R}^{n^r \times m}$, $x_0^r \in \mathbb{R}^{n_r}$.
Then $\Sigma_r=\SwitchSysLin[r]$
 is span-reachable, and has the same input-output
 map as $\Sigma$.
\end{Procedure}
Intuitively, $\Sigma_r$ is obtained from $\Sigma$ by restricting the
dynamics and the output map of $\Sigma$ to the space $\IM R(\Sigma)$. 
\begin{Procedure}[Observability reduction]
\label{LSSobs}
 Assume that $\ker O(\Sigma)=n-n^o$ and let $b_1,\ldots,b_n$ be a
 basis in $\mathbb{R}^n$ such that $b_{n^{o}+1},\ldots,b_{n}$ span
 $\ker O(\Sigma)$.  In this new basis, $A_q$,$B_q$, $C_q$ and $x_0$ can be rewritten as 
 \[ A_{q}=\begin{bmatrix} A_{q}^o, & 0 \\ A^{'}_{q}, & A^{''}_{q} \end{bmatrix},
    C_{q}=\begin{bmatrix} C_q^o, & 0 \end{bmatrix},
    B_{q}=\begin{bmatrix} B_{q}^o \\ B_{q}^{'} \end{bmatrix},
    x_0=\begin{bmatrix} x_0^o \\ x_{0}^{'} \end{bmatrix}
 \]
where $A^o_{q} \in \mathbb{R}^{n^o \times n^o}, B_q^o \in \mathbb{R}^{n^o \times m}$,
$C_q^o \in \mathbb{R}^{p \times n^o}$ and $x_0^o \in \mathbb{R}^{n_o}$.
Then the \BLSS $\Sigma_o=\SwitchSysLin[o]$
is observable and its input-output map is the same as that of
$\Sigma$. If $\Sigma$ is span-reachable, then so is $\Sigma_o$.
\end{Procedure}
Intuitively, $\Sigma_o$ is obtained from $\Sigma$ by merging any two
states $x_1$, $x_2$ of $\Sigma$, for which $O(\Sigma)x_1=O(\Sigma)x_2$.
The latter is equivalent to 
 $y_{\Sigma}(x_1,w)=y_{\Sigma}(x_2,w)$, $\forall w \in \HYBINP^{+}$. 
\begin{Procedure}[Minimization]
\label{LSSmin}
 First transform $\Sigma$ to a span-reachable \BLSS $\Sigma_r$ and
 then transform $\Sigma_r$ to an observable \BLSS $\Sigma_m=(\Sigma_r)_o$.
Then $\Sigma_m$ is a minimal realization of the input-output map of
$\Sigma$.
\end{Procedure}
The correctness of Procedures \ref{LSSreach},\ref{LSSobs} and \ref{LSSmin} are proved in \S \ref{sect:proof}, using the theory of formal power series.
Note that the correctness of Procedure \ref{LSSobs} and of Procedure
\ref{LSSreach} (in case of $x_0=0$) has already been 
shown by a direct proof in \cite{Sun:Book}. 
\begin{Example}
\label{example1}
 Let $\Sigma=\SwitchSysLin$ with
 $Q=\{1,2\}$, $n=3$, 
 \( x_0=\begin{bmatrix} 0 & 1 & 0 \end{bmatrix}^{T} \), 
 \[
    \begin{split}
     & A_{1}=\begin{bmatrix}
                      0 & 1 & 0  \\
                      0 & 0 & 1 \\
                      0 & 0  & 1 
                     \end{bmatrix}  \mbox{,\ }
     B_{1} =\begin{bmatrix}
                     0 \\ 0 \\ 0
                     \end{bmatrix} \mbox{, \ }
     C_{1}=\begin{bmatrix} 1 & 0 & 0 \end{bmatrix} \\
     & A_{2} =\begin{bmatrix}
                      0 & 1  & 0\\
                      0 & 1 & 1 \\
                      0 & 0 & 1 
                     \end{bmatrix}  \mbox{,\ }
     B_{2}=\begin{bmatrix}
                     0 \\ 1 \\ 0
                     \end{bmatrix} \mbox{, \ }
     C_{2} =\begin{bmatrix} 0 & 0 & 1 \end{bmatrix} \\
    \end{split}
  \]
  This system is observable, but it is not 
  span-reachable. In order to see observability, notice that
  the sub-matrix 
  $\begin{bmatrix} C_1^T & (C_1A_1)^T & C_2^T \end{bmatrix}^T$ of
  $O(\Sigma)$ is of rank $3$.
  In order to see that $\Sigma$ is not span-reachable, notice 
  that if $(x,y,z)^T$ is a column of $R(\Sigma)$, then $z=0$.
  Hence $\dim R(\Sigma) \le 2$.

  Using Procedure \ref{LSSmin}, we can transform
  $\Sigma$ to the minimal
  realization 
  \[ \Sigma_m=\SwitchSysLin[m] \]
   of $y_{\Sigma}$:
  $Q=\{1,2\}$, $n^m=2$,
  \( x^m_0=\begin{bmatrix} 1, & 0 \end{bmatrix}^T \) and
 \[ 
   \begin{split}
    & A^m_1=\begin{bmatrix} 0 &  0 \\
         1 & 0
        \end{bmatrix}, 
   B^m_1 = \begin{bmatrix}  0 \\ 0 \end{bmatrix}, 
   C^m_1 = \begin{bmatrix} 0, & 1 \end{bmatrix} \\
  \end{split}
\]
\[ \begin{split}
   & A^m_2 = \begin{bmatrix} 1  & 0 \\
           1  & 0
          \end{bmatrix}, 
    B^m_2 =\begin{bmatrix} 1 \\  0 \end{bmatrix}, 
    C^m_2 =\begin{bmatrix}  0, &  0 \end{bmatrix} \\
 \end{split}
\]
 Using \cite{Hof},
 it is easy to see that neither $(A^m_1,B^m_1,C^m_1,x_0^m)$ nor
 $(A^m_2,B^m_2,C^m_2,x_0^m)$ are minimal.
\end{Example}

\section{Main results on existence of a realization}
\label{sect:main_results:ex}
 We present the necessary and sufficient conditions for
 the existence of a \LSS\ realization  for an input-output map.
 In the sequel, \emph{$f$ denotes a map of the form \eqref{io_map}.}
 To this end, we need the notion of the Hankel-matrix  and
 Markov-parameters of  an input-output map. 
 More precisely, we proceed as follows.
 First, we define the notion of Markov parameters of $f$ and
 use them to define the Hankel-matrix of $f$. We then use the Hankel-matrix
 to formulate conditions for existence of a \BLSS realization of $f$.
 To this end, we need the following notation.
\begin{Notation}
 In the sequel, we identify 
 any element $w=(q_0,u_0)\cdots (q_t,u_t) \in \HYBINP^{+}$
 with the pair of sequences $(v,u)$, $v \in Q^{+}$,
 $u \in (\mathbb{R}^m)^{+}$, $v=q_0\cdots q_t$ and $u=u_0 \cdots u_t$.
\end{Notation} 
\begin{Notation}
 Consider the input-output map $f$. 
 For each word $v \in Q^{+}$ of length $|v|=t > 0$ define 
\( f_{v}:(\mathbb{R}^m)^{t} \rightarrow \mathbb{R}^p \)
as
\begin{equation}
 \label{sect:io:eq0}
  \begin{split}
   & f_{v}(u)=f((v,u)). 
  \end{split}
\end{equation}
\end{Notation}
 Now we are ready to define the Markov-parameters of an
 input-output map.
\begin{Definition}[Markov-parameters]
\label{sect:io:def0}
 Denote $Q^{k,*}=\{ w \in Q^{*} \mid |w| \ge k\}$.
 Define the maps $S_0^{f}:Q^{1,*} \rightarrow \mathbb{R}^p$ and
 $S_j^f:Q^{2,*} \rightarrow \mathbb{R}^{p}$, $j=1,\ldots,m$
  as follows; for any $v \in Q^{*}$, $q,q_0\in Q$,
\begin{equation}\label{eq:MarkovParam}
		\begin{split}
		& S_0^{f}(vq)=f_{vq}(0, \ldots, 0) \mbox{ and } \\
		& S^f_{j}(q_0vq)=f_{q_0vq}(e_j,0, \ldots, 0) - f_{q_0vq}(0,\ldots,0), 
		\end{split}
	\end{equation}
with $e_j\in \Re^m$ is the vector with $1$ as its $j$th entry and zero everywhere else. The collection of maps $\{S_j^f\}_{j=0}^{m}$ is 
 called the \emph{Markov-parameters} of $f$. 
\end{Definition}
The function $S_0^f$ can be viewed as the \emph{initial state-response} and
the functions $S_j^f$, $j=1,\ldots,m$ can be viewed as
\emph{input responses}. The interpretation of $S_0^f$, $S_j^f$
will become more clear 
 after we define
the concept of a \emph{generalized convolution representation}.
Note that the values of the Markov-parameters can be obtained 
from the values of $f$, i.e. by means of
input-output experiments.
 \begin{Notation}[Sub-word]
\label{not:subword}
  Consider the sequence $v=q_0\cdots q_t \in Q^{+}$,
  $q_0,\ldots, q_t \in Q$, $t \ge 0$.
 For each $j,k \in \{0,\ldots,t\}$, define
 the word $v_{j|k} \in Q^{*}$ as follows; 
 if $j> k$, then $v_{j|k}=\epsilon$, if $j=k$, then
 $v_{j|j}=q_j$ and if
 $j < k$, then $v_{j|k}=q_jq_{j+1}\cdots q_k$.
 That is, $v_{j|k}$ is the sub-word of $v$ formed by the letters
 from the $j$th to the $k$th letter.
 \end{Notation}
 \begin{Definition}[Convolution representation]
 \label{sect:io:def1}
  The input-output map $f$ has a \emph{generalized convolution
  representation (abbreviated as \GCR)}, if for all 
  $w=(v,u) \in \HYBINP^{+}$, $v=q_0\cdots q_t$,
  $u=u_0\cdots u_t$, $q_0,\ldots, q_t \in Q$, $u_0,\ldots u_t \in \mathbb{R}^m$, 
  $f(w)$ can be expressed via the Markov-parameters of $f$ as follows.
  \begin{equation*}
    \begin{split}
     & f(w)=S_0^f(v_{0|t-1}  q_t) 
      + \sum_{k=0}^{t-1} S^f(q_k  v_{k+1|t-1}  q_t)u_k
    \end{split}
  \end{equation*}
where $S^f(w)=\big[\begin{matrix}S_1^f(w), & S_2^f(w), &  \ldots,  & S_m^f(w) \end{matrix}\big] \in \Re^{p\times m}$ for all $w \in Q^{*}$.
 \end{Definition}
 \begin{Remark}
\label{rem:Unicity-of-from-MP}
  If $f$ has a \GCR, then
  the Markov-parameters of $f$ determine $f$ uniquely.
 \end{Remark}
  The motivation for introducing \GCR{s} is that existence of a \GCR\ is
  a necessary condition for realizability by \SLSS. 
  More precisely, the following holds.
 \begin{Lemma}
 \label{sect:io:lemma1}
  The map $f$ is  realized by the
  \LSS\  $\Sigma$ if and only if $f$ has a \GCR\ and 
  for all $v \in Q^{*}$, $q,q_0 \in Q$,
  \begin{equation}
  \label{sect:io:lemma1:eq1}
	\begin{split}
     & S_0^f(vq)=C_{q}A_{v}x_0 \mbox{ and} \\
     & S^f_{j}(q_0vq)=C_qA_vB_{q_0}e_j, \: j=1,\ldots,m. 
      \end{split}
  \end{equation}
 \end{Lemma}
  The proof of Lemma \ref{sect:io:lemma1} can be found in
  Appendix \ref{appA}.
  From Lemma \ref{sect:io:lemma1} it follows that
  if $f$ is realizable by a \LSS, then 
  the values of $S_0^f$ and $S_j^f$, $j=1,\ldots,m$ can be 
  expressed as products of matrices. 
  Moreover, $S_0^f$ corresponds to the part of the response which
  depends on the initial state,
  and $\{S_j^f\}_{j=1}^{m}$ encodes
  the response from the zero initial state.

  We can draw the following analogy with the linear case \S \ref{sect:lin}.
  Existence of a \GCR\ is analogous to the requirement
  that the input-output map is of the form \eqref{eq:LinearSys-Markov}.
  The Markov-parameter
  $S_0^f(vq)$ corresponds to the vector $K_{|v|}$, and the vector
  $S^f_{j}(q_0vq)$ corresponds to the $j$th column of the matrix $H_{|v|}$.
  Finally, if $f$ can be realized by a \LSS, then the Markov-parameters can be expressed as products of matrices \eqref{sect:io:lemma1:eq1}. This
  is analogous to the linear case, where $K_t=CA^tx_0$ and
  $H_t=CA^tB$ holds for $t \ge 0$, if $(A,B,C,x_0)$ is a realization of
  the input-output map.
  In fact,if $Q=\{1\}$, i.e. we are dealing with linear systems,
  then $S_0^f(vq)=K_{|v|}$, $S_j^f(q_0vq)$ is the $j$th column of
  $H_{|v|}$ and the \GCR\ is the representation of the form 
  \eqref{eq:LinearSys-Markov}, and the right-hand sides of \eqref{sect:io:lemma1:eq1} becomes $CA^{|v|}x_0$, $CA^{|v|}Be_j$, where
  $C=C_1,A=A_1, B=B_1$.

Next, we define the concept of a Hankel-matrix. 
Similarly to the linear case, the entries of
the Hankel-matrix are formed by the Markov parameters. 
For the definition of the Hankel-matrix of $f$,
we will use lexicographical ordering on the
set of sequences $Q^{*}$.
\begin{Remark}[Lexicographic ordering]
\label{rem:lex}
 Recall that $Q=\{1,\ldots,\QNUM\}$. We define a lexicographic
 ordering $\prec$ on $Q^{*}$ as follows.
 For any $v,s \in Q^{*}$, $v \prec s$ if either
 $|v| < |s|$ or $0 < |v|=|s|$, $v \ne s$ and for some $l \in \{1,\ldots, |s|\}$,
 $v_l < s_l$ with the usual ordering of integers and
 $v_i=s_i$ for $i=1,\ldots, l-1$. Here $v_i$ and $s_i$ denote the $i$th letter
 of $v$ and $s$ respectively.
 Note 
 that $\prec$ is a complete ordering and
 $Q^{*}=\{v_1,v_2,\ldots \}$ with $v_1 \prec v_2 \prec \ldots $.
 Note that $v_1 =\epsilon$ and for all $i \in \mathbb{N}$, $q \in Q$,
 $v_i \prec v_iq$.
\end{Remark}
In order to simplify the definition of a Hankel-matrix,
we introduce the notion of a combined Markov-parameter.
\begin{Definition}[Combined Markov-parameters]
 A combined Markov-parameter $M^f(v)$ of $f$ indexed by 
 the word
 $v \in Q^{*}$ is the following
 $p\QNUM \times (\QNUM m+1)$ matrix 
 \begin{equation}
   \label{main_results:lin:arb:pow2}
     M^f(v) = \begin{bmatrix}
             S^f_0(v1), & S^f(1v1), & \cdots, & S^f(\QNUM v1) \\
             S^f_0(v2), & S^f(1v2), & \cdots, & S^f(\QNUM v2) \\
             \vdots   & \vdots     & \cdots & \vdots \\
             S^f_0(v\QNUM), & S^f(1v\QNUM), & \cdots, & S^f(\QNUM v \QNUM)
             \end{bmatrix}
\end{equation}
 where for any $w \in Q^{+}$, $|w| > 2$,
 $S^f(w)=\begin{bmatrix} S^f_{1}(w), & S^f_{2}(w), & \ldots, & S^f_{m}(w) \end{bmatrix}$.
\end{Definition}
\begin{Definition}[Hankel-matrix] 
\label{main_result:lin:hank:arb:def}
 Consider the lexicographic ordering $\prec$ of $Q^{*}$ from
 Remark \ref{rem:lex}.
 Define the Hankel-matrix $H_f$ of $f$ as the following
 infinite matrix
 \[ H_f = \begin{bmatrix}
     M^f(v_1v_1), & M^f(v_2v_1), & \cdots, & M^f(v_kv_1), & \cdots \\
     M^f(v_1v_2), & M^f(v_2v_2), & \cdots, & M^f(v_{k}v_2), & \cdots, \\
     M^f(v_1v_3)
     & M^f(v_2v_3), & \cdots, & M^f(v_{k}v_3) & \cdots \\
     \vdots   
      & \vdots   & \cdots & \vdots & \cdots 
   \end{bmatrix},
 \]
i.e. the $p\QNUM \times (m\QNUM+1)$ block of $H_f$ 
in the block row $i$ and block column $j$
equals the combined Markov-parameter $M^f(v_jv_i)$ of $f$.
The rank of $H_f$, denoted by $\Rank H_f$, 
is the dimension of the linear span of its columns.
\end{Definition}
 The Hankel-matrix of $f$ can also be viewed as a matrix
 rows and columns of which are indexed by words from $Q^{*}$.
\begin{Remark}[Alternative definition of the Hankel-matrix]
\label{alt:hank}
 Notice that every row index $0 < l \in \mathbb{N}$ of $H_f$ 
 can be identified with a tuple
 $(v,i)$, $i=1,\ldots,p\QNUM$ and
 $v \in Q^{*}$ as follows;
 $v=v_r$, i.e. $v$ is the $r$th element of $Q^{*}$, 
 for some $0 < r \in \mathbb{N}$ such that
 $l=(r-1)\QNUM p+i$. In fact the identification above is
  a one-to-one mapping.

   Similarly, every column index $0 < k \in \mathbb{N}$ can be
   identified with  a pair $(w,j)$ where
      $w \in Q^{*}$, $j \in J_f = \{0\} \cup Q \times \{1,\ldots,m\}$, 
      where $w=v_r$, i.e. $w$ is the $r$th element of $Q^{*}$
      for some $r \in \mathbb{N}$ such that
      $k=(r-1)(m\QNUM+1)+i$ for some integer $i=1,\ldots,m\QNUM+1$, and
      if $i=1$ then $j=0$ and if $i=m(q-1)+z+1$ for some
      $q \in Q$ and $z=1,\ldots,m$, then $j=(q,z)$.
      This identification is one-to-one.

      Using the identification of row and column indices
      outlined above, we can view $H_f$ as a matrix, rows of
      which are indexed by $(v,i)$, $v \in Q^{*}$,
      $i=1,\ldots,p\QNUM$, and columns of which are indexed by
      $(w,j)$, $w \in Q^{*}$, $j \in J_f$.
      The entry $\left[H_{f}\right]_{(v,i),(w,j)}$ of $H_f$ indexed by row index $(v,i)$ and column
      index $(w,j)$ is the $i$th entry of the $r$th column
      of $M^f(wv)$, where $r=1$, if $j=0$ and $r = m(q-1)+z+1$ if $j=(q,z)$. In other words, 
  \[ 
   \begin{split}
   & \left[H_{f}\right]_{(v,i),(w,(q,z))}=\big[S^f_{z}(qwv\alpha_i)\big]_{l}, \\
  & \left[H_{f}\right]_{(v,i),(w,0)}=\big[S_0^f(wv\alpha_i)\big]_{l} 
   \end{split}
   \]
where $\alpha_i=K+1$ with $K$ and $l$ defined from $i$ by the decomposition $i=pK+l$,  $K=0,1,\ldots, \QNUM-1$, $l=1,\ldots, p$. 
Here, $\left[a\right]_l$ denotes the $l$th entry of a vector $a$. 
\end{Remark}
It is not difficult to see that for $Q=\{1\}$, $H_f$ is the same as
the Hankel-matrix defined in \S \ref{sect:lin}.
The main result on realization theory of
\SLSS\ can be stated as follows.
  \begin{Theorem}
  \label{sect:real:theo2}
      The map $f$ has a realization by  a \LSS\  if and
      only if $f$ has a \GCR\ 
      and 
     $\Rank H_{f} < +\infty$. A minimal realization of $f$
     can be constructed from $H_{f}$ (see Procedure \ref{LSSfromHankel})
     and any minimal \LSS\ 
     realization of $f$ has dimension $\Rank H_f$.
\end{Theorem}
\begin{Procedure}
\label{LSSfromHankel}
 If $\Rank H_f = n < +\infty$, then a \BLSS $\Sigma_f$ 
of the form \eqref{lin_switch0} can be constructed from $H_f$ as follows.  
Choose a basis in the
column space of $H_f$. 

In this basis, let
$x_0$ be the coordinates of the first column of $H_f$.
For each $l=1,\ldots,m$, the $l$th column of $B_q$, $q \in Q$
is formed by 
coordinates of the
$m(q-1)+l+1$th column of $H_f$.
Let $C_q$, $q \in Q$ be the matrix of the linear map which maps
every column to the vector formed by its rows indexed
by $p(q-1)+1,p(q-1)+2,\ldots, pq$.
Define $A_q$, $q \in Q$ as the matrix of the linear map
which maps the $r$th column of the block column
$(M(v_jv_i))_{i=1}^{\infty}$ to the $r$th column of the block column
$(M(v_jqv_i))_{i=1}^{\infty}$, for each $j=1,2,\ldots, $ and
$r=1,2,\ldots, (\QNUM m+1)$.

Alternatively, using Remark \ref{alt:hank} we can describe $\Sigma_f$
as follows. The initial state $x_0$ is formed
by the coordinates of the column of $H_f$ indexed by $(\epsilon,0)$.
The $l$th column of $B_q$, $q \in Q$ is formed by the coordinates
of the column of $H_f$ indexed by $(\epsilon, (q,l))$, $l=1,\ldots,m$.
The matrix $C_q$, $q \in Q$ is the matrix of the linear map
which maps each column of $H_f$ to the vector formed by its rows which
are indexed by $(\epsilon, p(q-1)+1),\ldots, (\epsilon, pq)$.
Finally, $A_q$ is the matrix of the map which maps each
column indexed by $(w,j)$ to the column indexed by $(wq,j)$,
$w \in Q^{*}$, $j \in J_f$.
\end{Procedure}
 Notice that for $Q=\{1\}$, Theorem \ref{sect:real:theo2} implies
 Theorem \ref{lin:theo:ex}, and Procedure \ref{LSSfromHankel}
 reduces to Procedure \ref{lin:hank}.
\begin{Example}
\label{example2}
 Consider a SISO input-output map $f$ such that for any $v \in Q^{+}$, 
$|v|=t$,
 \[ 
   f_{v}(u_1,\ldots,u_t)=\left\{\begin{array}{rl}
    1+\sum_{j=1}^{t-2} u_j & \mbox{ if }  t > 1 \mbox{ and } 
    v=2^{t-1}1 \mbox{ or } \\ 
    &  v=2^{t-2}11,  \\
    0 & \mbox{ otherwise }
   \end{array}\right.
 \]
 Hence, the Markov-parameters of $f$ are as follows
 \[ 
   \begin{split}
     & S_0^f(v)=
             \left\{\begin{array}{rl}
             1 & \mbox{ if } t > 1 \mbox{ and } v=2^{t-1}1 \mbox{ or } v=2^{t-2}11  \\
             0 & \mbox{ otherwise } \\
	   \end{array}\right. \\
     & S_1^f(v)=
             \left\{\begin{array}{rl}
             1 & \mbox{ if } t > 2 \mbox{ and } v=2^{t-1}1 \mbox{ or } v=2^{t-2}11 \\
             0 & \mbox{ otherwise } 
             \end{array}\right.
  \end{split}
\]
 It is easy to check that $\Sigma$ from Example \ref{example1} 
 satisfies \eqref{sect:io:lemma1:eq1} from Lemma \ref{sect:io:lemma1},
 hence $\Sigma$ is a realization of $f$.

  Consider the Hankel-matrix $H_f$ of $f$. It is easy to see that
  the set of columns of $H_f$ contains two elements: $b_1$ and $b_2$.
  The entries of $b_1$ equal $1$, if indexed by $(v,1)$ with  
  $|v| > 0$ and $v=2^{|v|}$ or $v=2^{|v|-1}1$ and are zero otherwise.
  The only non-zero entry of $b_2$ is $1$ and it is indexed by $(\epsilon,1)$. 
  Applying Procedure \ref{LSSfromHankel} to our example, and taking
  $(b_1,b_2)$ as a basis of $\IM H_f$, we obtain a
  \BLSS of the form \eqref{lin_switch0} which coincides with 
  $\Sigma_m$ from Example \ref{example1}.  
 
  Indeed, since the column of $H_f$
  indexed by $(\epsilon, (1,1))$ is zero, and the column indexed by
  $(\epsilon,0)$ and $(\epsilon,(2,1))$ is $b_1$, we get $B_1=0$,
  $B_2=x_0=(1,0)$. Since the entries of any column indexed by $(\epsilon,2)$
  are zero, we get $C_2=0$. Since the entries of $b_1$ and $b_2$ indexed
  by $(\epsilon,1)$ are $1$, we get $C_1=(1,1)^T$. 
  Note that
  if the column of $H_f$ indexed by $(w,j)$ equals $b_1$, then
  the column indexed by $(w1,j)$ equals $b_2$, the column indexed by
  $(w2,j)$ equals $b_1+b_2$. If the column indexed by $(w,j)$ equals $b_2$, then 
  the column indexed by $(w1,j)$ and $(w2,j)$ are both zero. Hence, if
  $A_1$ and $A_2$ are viewed as linear maps on $\IM H_f$, then
  $A_1b_1 = b_2$, $A_1b_2=0$, $A_{2}b_2=0$, $A_2b_1=b_1+b_2$. In other words,
  the matrices $A_1$ and $A_2$ are precisely the same as the matrices
  $A_1^m$ and $A_2^m$ from Example \ref{example1}.
\end{Example}

 Note that once the Markov-parameters are defined,
 the definition of Hankel-matrix presented above 
 coincides with that of the continuous-time case.
 As a consequence, we can repeat
 the realization algorithm described 
 in \cite[Algorithm 1]{MP:PartReal} for \BSLSS.
 Moreover, \cite[Theorem 4]{MP:PartReal} holds for
 \BSLSS.  For the sake of completeness, below we state the
 realization algorithm and its correctness explicitly
 for \BSLSS.
\begin{Definition}[$H_{f,L,M}$ sub-matrices of $H_{f}$]
 \label{sect:main_result:sub-hank:def1}
  For $L,M \in \mathbb{N}$ 
  define the integers 
  $I_L=\mathbf{N}(L)p\QNUM$ and $J_M=\mathbf{N}(M)(m\QNUM+1)$.
  Denote by $H_{f,L,M} $ the following upper-left $I_L \times J_M$ 
  sub-matrix of $H_{f}$, 
 \[ 
    \begin{bmatrix}
     M^f(v_1v_1), & M^f(v_2v_1), & \cdots, & M^f(v_{\mathbf{N}(M)}v_1) \\
     M^f(v_1v_2), & M^f(v_2v_2), & \cdots, & M^f(v_{\mathbf{N}(M)}v_2) \\
     \vdots   
      & \vdots   & \cdots & \vdots & \\
     M^f(v_1v_{\mathbf{N}(L)}), & M^f(v_2v_{\mathbf{N}(L)}), & \cdots, & M^f(v_{\mathbf{N}(M)}v_{\mathbf{N}(L)}) 
   \end{bmatrix}.
 \]
 \end{Definition}
   \begin{algorithm}
   \caption{
    \newline
    \textbf{Inputs:} Hankel-matrix $H_{f,N,N+1}$.  
    \newline
    \textbf{Output:} \LSS\ $\Sigma_N$
    }
    \label{alg0} 
    \begin{algorithmic}[1]
    \STATE
        Compute the decomposition
        $H_{f,N,N+1} = \mathbf{O}\mathbf{R}$ such that 
        $\mathbf{O} \in \mathbb{R}^{I_{N} \times n}$ and
        $\mathbf{R} \in \mathbb{R}^{n \times J_{N+1}}$ and 
        $\Rank \mathbf{R} =\Rank \mathbf{O} = n$.
   \STATE
     Consider the decomposition
    \[ \mathbf{R}=\begin{bmatrix} \mathbf{C}_{v_1}, & \mathbf{C}_{v_2}, & \ldots, & \mathbf{C}_{v_{\mathbf{N}(N+1)}} \end{bmatrix} 
     \] 
     such that $\mathbf{C}_{v_i} \in \mathbb{R}^{n \times (\QNUM m+1)}$.
     Define $\overline{\mathbf{R}}, \mathbf{R}_q \in \mathbb{R}^{n \times J_N}$, $q \in Q$ as follows
    \[
      \begin{split}
       & \overline{\mathbf{R}}=
       \begin{bmatrix} \mathbf{C}_{v_1}, & \mathbf{C}_{v_2}, & \ldots, & \mathbf{C}_{v_{\mathbf{N}(N)}} \end{bmatrix}  \\
      & \mathbf{R}_q=
       \begin{bmatrix} \mathbf{C}_{v_1q}, & \mathbf{C}_{v_2q}, & \ldots, & \mathbf{C}_{v_{\mathbf{N}(N)q}} \end{bmatrix}  \\
      \end{split}
     \]

    \STATE
       
       Construct $\Sigma_N$ of the form (\ref{lin_switch0}) such that
          \begin{eqnarray}
              & & 
		\begin{bmatrix} x_0, B_1,\ldots, B_{\QNUM}
                \end{bmatrix} = \nonumber  \\
             & & \mbox{the first $m\QNUM+1$ columns of $\mathbf{R}$} 
          \label{alg0:eq-2} \\
             & &  \begin{bmatrix} C_1^T, & C_2^T, & \ldots, & C_{\QNUM}^T 
\end{bmatrix}^T = \mbox{ the first $p\QNUM$ rows of $\mathbf{O}$} 
             \label{alg0:eq-1}  \\
	& & \forall q \in Q:  A_q=\mathbf{R}_{q}\overline{\mathbf{R}}^{+}, 
         \label{alg0:eq1}
          \end{eqnarray}
      where $\overline{\mathbf{R}}^{+}$ is the Moore-Penrose pseudoinverse
      of $\overline{\mathbf{R}}$.
    \STATE
       Return $\Sigma_N$
    \end{algorithmic}
   \end{algorithm}
\begin{Remark}[Implementation]
 One way to compute the factorization 
 $H_{f,N,N+1}=\mathbf{O}\mathbf{R}$ is as follows.
 If $H_{f,N,N+1}=U\Sigma V$ is the SVD 
  decomposition of $H_{f,N,N+1}$, then define
 $\mathbf{O}=U\Sigma^{1/2}$ and $\mathbf{R}=\Sigma^{1/2}V$.
\end{Remark}
  \begin{Theorem}
  \label{part_real_lin:theo1}
  If $\Rank H_{f,N,N}=\Rank H_{f}$, then 
  the algorithm returns a minimal realization of $f$.
  The condition $\Rank H_{f,N,N}=\Rank H_{f}$ holds for a given $N$, if
  there exists an \LSS\  realization $\Sigma$ of $f$ such that
  $\dim \Sigma \le N+1$.
\end{Theorem}
The proof of Theorem \ref{part_real_lin:theo1} can be found in
\S \ref{sect:proof}.
\begin{Remark}[Computation of $H_{f,N,N}$]
 Note that $H_{f,N,N}$ can be computed from the responses of $f$.
 However, in principle, 
 the computation of $H_{f,N,N}$ requires an exponential number
 of input/output experiments involving different switching sequences.
 This is clearly not very practical. It would be more practical
 to build $H_{f,N,N}$ based on the response of $f$ to a single
 switching sequence. Preliminary results on the latter approach
 can be found in \cite{MPLB:Pers}.
  A detailed discussion of this approach
 goes beyond the scope of this paper.
\end{Remark}

\section{ Formal Power Series }
\label{sect:pow}
In this section we present an overview of the necessary results on
formal power series.
The material of the section is an extension of the classical
theory of \cite{Reut:Book,Son:Real}, for the proofs of the results
of this section
see \cite{MP:Phd,MP:BigArticlePartI}.

 Let $X$ be a finite set, which we refer to as the alphabet.
  A \emph{formal power series} $S$ with coefficients in 
  $\mathbb{R}^{\POWO}$ is a map
  \[ S: X^{*} \rightarrow \mathbb{R}^{\POWO} \]
   We denote by
   $\mathbb{R}^{\POWO} \ll   X^{*}\gg$ the set of all such maps.
   Let $J$ be an arbitrary 
   (possibly infinite) set. 
   A \emph{family of formal power series in 
    $\mathbb{R}^{\POWO}\ll X^{*}\gg$ indexed by $J$}, abbreviated as
    \RFSF is a collection 
   \begin{equation}
   \label{fam:pow1}   
   \Psi=\{ S_{j} \in \mathbb{R}^{\POWO}\ll X^{*} \gg \mid j \in J\}.
  \end{equation}
 In the sequel \emph{$\Psi$ denotes  a \RFSF of the form (\ref{fam:pow1})}.
    Notice that we do not require 
    $S_{j}$, $j \in J$ to be all
    distinct , i.e. $S_{l}=S_{j}$ for some indices
    $j,l \in J$, $j \ne l$ is allowed.

   Let $J$ be an arbitrary set and let $\POWO >0$.  A 
   \emph{$\POWO$-$J$ rational representation 
   over the alphabet $X$}
   is a tuple
   \begin{equation}
   \label{repr:def0}
    R=(\mathcal{X},\{A_{\sigma}\}_{\sigma \in X},B,C)
   \end{equation}
   where $\mathcal{X}$ is a finite-dimensional 
   vector space over $\mathbb{R}$, for each 
   $\sigma \in X$, 
   $A_{\sigma}:\mathcal{X} \rightarrow \mathcal{X}$ is a 
   linear map, $C:\mathcal{X} \rightarrow \mathbb{R}^{\POWO}$
   is a linear map, and $B=\{ B_{j} \in \mathcal{X} \mid j \in J\}$
   is a family of elements of $\mathcal{X}$ indexed
   by $J$. If $\POWO$ and $J$ are clear from the context we will
   refer to $R$ simply as a \emph{rational representation}.
   We call $\mathcal{X}$ the 
   \emph{ state-space }, $A_{\sigma}$, $\sigma \in X$
   the \emph{state-transition maps}, and
   $C$ the \emph{readout map} of $R$.
   The family $B$ is called the \emph{family of initial states of $R$}.
   The dimension $\dim \mathcal{X}$ of the state-space is called the 
   \emph{dimension} of $R$ and
   it is denoted by $\dim R$. 
   If $\mathcal{X}=\mathbb{R}^{n}$, then we identify
   the linear maps $A_{\sigma}$, $\sigma \in X$ and
   $C$ with their matrix representations in the standard
   Euclidean bases,  and we call them the \emph{state-transition matrices} and the
   \emph{readout matrix} respectively.

The $\POWO-J$ representation $R$ from (\ref{repr:def0}) is said to be
  a \emph{representation of $\Psi$}, if 
   \begin{equation}
   \label{repr:def1}
   \forall j \in J, \forall w \in X^{*}: S_{j}(w)=CA_wB_j,
   \end{equation}
	where  Notation \ref{repr:not1} has been used. 
  We say that the family 
  $\Psi$ 
  is \emph{rational}, if there exists a $\POWO$-$J$ representation
  $R$ such that $R$ is a representation of $\Psi$.
%
 A representation $R_{min}$ of $\Psi$ is called \emph{minimal}
 if  for each representation $R$ of $\Psi$, 
 \( \dim R_{min} \le \dim R \).
 Define the subspaces  
  \begin{eqnarray}
  \label{sect:pow:reachobs:eq1}
  W_{R}  &=&  \SPAN\{ A_{w}B_{j} \in \mathcal{X} \mid w \in X^{*}, |w| < n, j \in J\} \\    
  \label{sect:pow:reachobs:eq2}
     O_{R} &=&  \bigcap_{w \in X^{*}, |w| < n} \ker CA_{w}. 
  \end{eqnarray}
  We will say that the representation $R$ is \emph{reachable}
  if $\dim W_{R}=\dim R$, and we will say that $R$ is
  \emph{observable} if $O_{R}=\{0\}$.
  Let $R=(\mathcal{X},\{ A_{\sigma} \}_{\sigma \in X},B,C)$,
  \( \widetilde{R}=(\widetilde{\mathcal{X}},\{ \widetilde{A}_{\sigma} \}_{\sigma \in X}, \widetilde{B}, \widetilde{C}) \) be two 
  $\POWO-J$ rational representations. 
  A linear map
  $\MORPH: \mathcal{X} \rightarrow \widetilde{\mathcal{X}}$ is
  called a \emph{representation morphism},  
  and is denoted by $\MORPH:R \rightarrow \widetilde{R}$,
  if 
  \begin{equation}
     \begin{array}{rcl}
     \label{repr:morph:eq2}
     \MORPH A_{\sigma}=\widetilde{A}_{\sigma}\MORPH, \forall \sigma \in X, 
     & \MORPH B_{j}=\widetilde{B}_{j}, \forall j \in J,  &
     C=\widetilde{C}\MORPH
   \end{array}
  \end{equation}
     If $\MORPH$ is bijective, then it is called a representation isomorphism.
     If $\MORPH$
     is an isomorphism, then $\widetilde{R}$ and $R$ are
     representations of the same \RFSF, and 
     $R$ is observable (reachable)
     if and only if $\widetilde{R}$ is observable (reachable).
  \begin{Remark}
  \label{repr:state_space:rem}
Let $R$ be a 
representation of $\Psi$ of the form (\ref{repr:def0}), and consider
a linear isomorphism $\MORPH:\mathcal{X} \rightarrow \mathbb{\mathbb{R}}^{n}$, 
$n=\dim R$. Then
\(  \MORPH R=(\mathbb{\mathbb{R}}^{n}, \{ \MORPH A_{\sigma}\MORPH^{-1} \}_{\sigma \in X}, \MORPH B, C\MORPH^{-1}) \), 
where $\MORPH B=\{ \MORPH B_{j} \in \mathbb{R}^{n} \mid j \in J\}$
is a representation of $\Psi$ and it is isomorphic to $R$.
The representation $\MORPH R$ is defined on an Euclidean space and
its state-transition and readout maps can be viewed as matrices.
\end{Remark}

    \begin{Definition}[Hankel-matrix]
    \label{sect:pow:hank:def0}
    Define the \emph{Hankel-matrix} 
    $H_{\Psi}$ 
    of $\Psi$ as the
    infinite matrix, the rows of which are indexed
    by pairs $(v,i)$ where $v \in X^{*}$, $i=1,\ldots, \POWO$, and  
    the columns of which are indexed
    by $(w,j)$ where $w \in X^{*}$, $j \in J$. 
    The entry $\left[H_{\Psi}\right]_{(v,i),(w,j)}$ of $H_{\Psi}$ 
    indexed with the row index $(v,i)$ and the column index
    $(w,j)$ is defined as
    \begin{equation}
    \label{repr:eq3}
      \left[H_{\Psi}\right]_{(v,i)(w,j)}= \left[S_{j}(wv)\right]_{i}
    \end{equation}
    where $\left[S_{j}(wv)\right]_{i}$ denotes the $i$th entry of
    the vector $S_{j}(wv) \in \mathbb{R}^{\POWO}$.
    The rank of $H_{\Psi}$ is
    the dimension of the linear space spanned by the
    columns of $H_{\Psi}$, and it is denoted by $\Rank H_{\Psi}$.
 \end{Definition}   
 \begin{Theorem}[Existence and minimality, \cite{MP:Phd,MP:BigArticlePartI}]
 \label{sect:form:theo1}
 \begin{enumerate}
 \item
  The family $\Psi$ is rational,
   if and only if $\Rank H_{\Psi} <+\infty$.
 \item
 If $\Rank H_{\Psi} < +\infty$, then a minimal representation $R$ of $\Psi$ can be constructed from $H_{\Psi}$, see Procedure \ref{reprfromhankel}.
 \item
 Assume that $R_{min}$ is a representation of $\Psi$.
    Then $R_{min}$
    is a minimal representation of $\Psi$, if and only
    if  $R_{min}$ is reachable and observable.
    If $R_{min}$ is minimal, then $\Rank H_{\Psi}=\dim R_{min}$.
 \item
 All minimal representations of $\Psi$ are isomorphic.
 \item
  Any representation $R$ of $\Psi$ can be transformed to a 
  minimal representation $R_{min}$ of $\Psi$, see
  Procedure \ref{repr:constr:can}.
 \end{enumerate}
\end{Theorem}
 We conclude by presenting procedures for reachability and observability reduction, minimization of representations and construction of a representation from
 the Hankel-matrix.  
In the sequel, $R$ is a representation of $\Psi$ and $R$ is of the form \eqref{repr:def0}.
\begin{Procedure}[Repr. from Hankel-matrix, \cite{MP:Phd,MP:BigArticlePartI}]
\label{reprfromhankel}
  If $\Rank H_{\Psi} < +\infty$, then 
   \[ R_{\Psi}=(\IM H_{\Psi},\{A_{\sigma}\}_{\sigma \in X},B,C) \]
   is a representation of $\Psi$.
   Here, for each $\sigma \in X$, 
   $A_{\sigma}$ is the linear map which maps every column of
   $H_{\Psi}$ indexed by $(w,j)$ to the column indexed
   by $(w\sigma,j)$.
    The initial states are
    $B=\{ B_{j} \mid j \in J \}$, where
    $B_{j}$ is the column of
    $H_{\Psi}$ indexed by $(\epsilon,j)$, $j \in J$.
    Finally, $C$ is a linear map which maps every column of $H_{\Psi}$
    to the vector formed by those rows of this
    columns which are indexed by $(\epsilon,1),\ldots, (\epsilon,\POWO)$. Recall that $\mathbb{R}^{\POWO}$ is set of
   coefficients of the formal power series $S_j$ of $\Psi$, $j \in J$, i.e. $S_j:X^{*} \rightarrow \mathbb{R}^{\POWO}$.
\end{Procedure}
\begin{Procedure}[Reachability Reduction]
\label{repr:constr:reach}
 Assume $R$ is a representation of $\Psi$ and it is of the form \eqref{repr:def0}.
 Recall the definition of the reachable subspace $W_{R}$
 of $R$ from (\ref{sect:pow:reachobs:eq1}).
 Define the representation
\(    R_{r}=(W_{R},\{A_{\sigma}^{r}\}_{\sigma \in X},B^{r},C^{r}) \), 
 where for each $\sigma \in X$,
 $A_{\sigma}^{r}$ is the restriction of $A_{\sigma}$ to
 $W_{R}$, 
 $B^{r}=\{ B_{j} \in \mathcal{X} \mid j \in J \}=B$,
 and $C^{r}$ is the restriction of
 $C$ to $W_{R}$.
 Then $R_r$ is a reachable representation of $\Psi$.
\end{Procedure}
\begin{Procedure}[Observability Reduction]
\label{repr:constr:obs}
 Assume $R$ is a representation of $\Psi$ and it is of the form \eqref{repr:def0}.
 Recall from (\ref{sect:pow:reachobs:eq2}) the definition of
 the observability subspace $O_{R}$. 
Define the 
 representation
 \( R_{o}=(\mathcal{X}/O_{R_{r}}, \{\widetilde{A}_{\sigma}\}_{\sigma \in X}, \widetilde{B},\widetilde{C})  \).
 Here
 $\mathcal{X}/O_{R}$ is the quotient space of
 $\mathcal{X}$ with respect to $O_{R}$. Denote by
 $[x]$, $x \in \mathcal{X}$ the equivalence class of all those $y \in \mathcal{X}$ such that $x-y \in O_{R}$.
 Then
 $\widetilde{A}_{\sigma}[x]=[A_{\sigma}x]$, $\sigma \in X$, 
 $\widetilde{C}[x]=Cx$ for all $x \in \mathcal{X}$, and $\widetilde{B}=\{ \widetilde{B}_{j} \in \mathcal{X}/O_{R} \mid j \in J\}$ is such that
 $\widetilde{B}_{j}=[B_{j}]$, $j \in J$.
 Then $R_o$ is an observable representation of $\Psi$ and
 if $R$ is reachable, then so is $R_o$.
 \end{Procedure}
\begin{Procedure}[Minimization]
\label{repr:constr:can}
  A representation $R$ of $\Psi$ can be converted to a minimal
  representation as follows.
  Use Procedure \ref{repr:constr:reach} to obtain a
  reachable representation $R_{r}$. 
  Apply Procedure \ref{repr:constr:obs} to
  $R_r$ and  obtain the observable representation $R_{min}=(R_{r})_{o}$.
  Then  $R_{min}$ is a minimal representation of $\Psi$.
\end{Procedure}
 If $J$ is finite, then
 Procedures \ref{reprfromhankel}, \ref{repr:constr:reach}, \ref{repr:constr:obs},
 and \ref{repr:constr:can} can be implemented, see
 \cite{MP:Phd}.
 
More precisely, we can formulate a realization algorithm for
rational representations, \cite{MP:PartTechRep}.
Below we present slight extension of
the results of 
\cite{Son:Resp,Son:Real,isi:tac} on realization algorithms for
formal power series. 
The proofs of the results can be found in \cite{MP:Phd,MP:PartTechRep}.
We introduce the following notation.
Let $K, M  \in \mathbb{N}$.
\begin{equation}
 \begin{split}
 \mathbb{I}_{M}&=\{(v,i) \mid v \in X^{*}, |v| \le M,i=1,\ldots,p \} \\
 \mathbb{J}_{K}&=\{ (w,j) \mid j \in J, w \in X^{*}, |w| \le K \}
\end{split}
\end{equation}
 Intuitively, the elements of $\mathbb{I}_M$ (resp. $\mathbb{J}_K$) are those
 row (column) indices of $H_{\Psi}$, the $X^{*}$-valued component of 
 which is of length at most $M$ (resp. $K$).
\begin{Definition}
\label{part_real_pow:submatrix}
Define the matrix $H_{\Psi,M,K}$ as the matrix, rows of
 which are indexed by the elements of $\mathbb{I}_M$, columns of which are indexed
 by the
 elements of $\mathbb{J}_K$, and its entry $(H_{\Psi,M,K})_{(v,i),(w,j)}$
 indexed by the row index $(v,i) \in \mathbb{I}_M$ and the column index 
 $(w,j) \in \mathbb{J}_K$ is defined as
\( (H_{\Psi,M,K})_{(v,i),(w,j)}=(H_{\Psi})_{(v,i),(w,j)}=(S_{j}(wv))_{i}
\).
The rank of $H_{\Psi,M,K}$, denoted by $\Rank H_{\Psi,M,K}$, is the dimension of the linear space spanned by its columns.
\end{Definition}
That is, $H_{\Psi,M,K}$ is the sub-matrix of $H_{\Psi}$ formed
by the intersection of the columns indexed by the 
elements of  $\mathbb{J}_K$ 
and of the rows indexed by the elements of $\mathbb{I}_M$. If
$J$ is finite, then $H_{\Psi,M,K}$ is a \emph{finite matrix}.
\begin{Theorem}[Realization algorithm, \cite{MP:PartTechRep}]
\label{part_real_pow:theo}
   If
   \(  \Rank H_{\Psi,N,N}=\Rank H_{\Psi} \), 
   then the representation $R_{N}$, to be defined below,
    is a minimal representation of $\Psi$.
   If $\Rank H_{\Psi} \le N+1$, then  \( \Rank H_{\Psi,N,N}=\Rank H_{\Psi} \) holds.
   The representation $R_N$ is of the form
      (\ref{repr:def0}), with the
      state-space $\mathcal{X}=\IM H_{\Psi,N,N+1}$, and such that
      if we denote by $\COL_{w,j}$ the column of $\IM H_{\Psi,N,N}$
      indexed by $(w,j) \in \mathbb{J}_N$, then
      \[
      \begin{split}
      & \forall (w,j) \in \mathbb{J}_N: A_{\sigma}(\COL_{w,j})=\COL_{w\sigma,j} \\
      & \forall (w,j) \in \mathbb{J}_N:  
           C(\COL_{w,j})=\begin{bmatrix} \COL_{w,j}((\epsilon,1)), & \cdots, & \COL_{w,j}((\epsilon,p)) \end{bmatrix}^{T} \\
      & \forall j \in J: B_{j}=\COL_{\epsilon,j}
      \end{split}
      \]
     Here $C_{w,j}((\epsilon,i))$ is the entry of the column
     $\COL_{w,j}$ indexed by $(\epsilon,i)$.
     i.e.  it equals $(H_{\Psi,N,N})_{(\epsilon,i),(w,j)}$,
     $i=1,\ldots,p$.
\end{Theorem}

\section{Proof of the main results}
\label{sect:proof}
 The proof of the results on realization theory relies on
 the relationship between formal power series representations
 and \BSLSS state-space representations. This relationship
 is completely analogous to the one for linear switched systems in
 continuous time, \cite{MP:BigArticlePartI,MP:RealForm}.

 Consider an input-output map $f$ and assume that
  $f$ has a \GCR.
 Below we define the \emph{\RFSF $\Psi_{f}$ associated with $f$}. We
 also define the \emph{representation $R_{\Sigma}$ associated with 
 a \BLSS $\Sigma$} and a \emph{\BLSS $\Sigma_R$ associated with a rational representation $R$}. These notions allow us to relate \RFSF
 and input-output maps and to relate \BLSS with rational 
 representations. In turn, these correspondences enable us to
 translate the realization problem for \BLSS to the problem  of
 rationality of \NFSF.

We first define the \RFSF\ associated with $f$.
 To this end, recall the definition \eqref{eq:MarkovParam} of the Markov-parameters of $f$. 
\begin{Definition}[\RFSF\ associated with $f$]
  For each $q\in Q$, each index $j=1,\ldots, m$,  define
  the formal power series
  $\mathbb{S}_{q,j},\mathbb{S}_{0} \in \Re^{pD}\ll Q^{*} \gg$ 
  as follows; for each word $w \in Q^{*}$, discrete mode
  $q \in Q$ and index $j=1,\ldots,m$,
  \begin{equation}
  \label{sect:real:form-1}
  \begin{split}
    & \mathbb{S}_{(q,j)}(w )=\begin{bmatrix}
       (S^f_{j}(qw1))^{T}, & (S^f_{j}(qw2))^T, & \cdots, & (S^f_{j}(qw\QNUM ))^{T}
       \end{bmatrix}^{T},  \\ 
    & \mathbb{S}_{0}(w)=\begin{bmatrix}
                 (S_0^f(w1))^{T}, & (S_0^f(w2))^T, & \cdots, & (S_0^f(w\QNUM))^T
	       \end{bmatrix}^T.	 
    \end{split}
  \end{equation}	       
 Let  $J_{f}=\{0\} \cup \{ (q,l) \mid q \in Q, l=1,\ldots,m\}$ 
  and  define the \emph{\RFSF associated with $f$} by 
    \begin{equation}
    \label{sect:real:form-2}
    \Psi_{f}=\{ \mathbb{S}_{j} \in \mathbb{R}^{p \QNUM} \ll Q^{*}\gg \mid j \in J_{f} \}.
     \end{equation}
\end{Definition}
 Notice that 
 the values of $\mathbb{S}_{(q,j)}(w)$ and $S_0^f(w)$ 
 are obtained by stacking up
 the Markov-parameters of $S^f_{j}(qwi)$ and $S_0^f(wi)$ respectively, for $i=1,\ldots, \QNUM$.
  Next, we define the representation
  $R_{\Sigma}$ associated with $\Sigma$.
\begin{Definition}
  \label{sect:real:const1}
  Assume that $\Sigma$ is of the form (\ref{lin_switch0}).
  Define the 
  \emph{representation $R_{\Sigma}$ associated with $\Sigma$} as a 
  $p|Q|-J_f$ representation of the form \eqref{repr:def0}, where
  $J_f=\{0\} \cup Q \times \{1,\ldots,m\}$ and the following holds.
  \begin{itemize}
  \item
  The alphabet $X$ of $R_{\Sigma}$ is the set of discrete modes $Q$, and
  $\POWO=p|Q|$.
  \item
  The state-space $\mathcal{X}$ of $R_{\Sigma}$ is the same as that of
  $\Sigma$, i.e. $\mathcal{X}=\mathbb{R}^n$.
  For each $q \in Q$, the state-transition matrix $A_{q}$ of $R_{\Sigma}$ is identical to the
  matrix $A_{q}$ of $\Sigma$. 
  \item
  The $p|Q| \times n$ readout matrix $C$ 
   is obtained by vertically
   "stacking up" the matrices $C_{1},\ldots, C_{\QNUM}$, i.e.
   \[ C=\begin{bmatrix} 
     C_{1}^{T}, & C_{2}^{T}, & \cdots, & C_{\QNUM}^{T} \end{bmatrix}^{T} \in \mathbb{R}^{p\QNUM \times n}.
   \]
   \item
   \( B=\{ B_{j} \in \mathcal{X} \mid j \in J_{f}\} \), where
   $B_{0}=x_0$ and 
   $B_{(q,l)}$ is the $l$th column of the matrix $B_q$ of
   $\Sigma$.
  \end{itemize}
 \end{Definition}  
   The intuition behind the definition of $R_{\Sigma}$ is
   that
   we would like $R_{\Sigma}$ to be a representation
   of $\Psi_{f}$ if and only if
   (\ref{sect:real:pow:eq3}) holds.
   Then the $A_{q}$ matrices of the representation $R_{\Sigma}$
   should coincide with the $A_{q}$ matrices of $\Sigma$.
   The initial states of $R_{\Sigma}$
   should be formed by the vector $B_{0}$ (in order to generate
   $\mathbb{S}_{0}$), and $B_{q}e_{j}$ (in order to generate $\mathbb{S}_{(q,j)}$).
   Finally, the readout map $C$ should be formed
   by "stacking up" the matrices $C_{q}$.
    Next, we define a \BLSS $\Sigma_R$ based on a representation $R$.
   \begin{Definition}
   \label{sect:real:const2}
   Consider a $p|Q|-J_f$ representation $R$ of the form
  \eqref{repr:def0}, over the alphabet $X=Q$ with $\POWO=p|Q|$.
    If  $\mathcal{X}=\mathbb{R}^{n}$ does not
    hold, then replace $R$ with the isomorphic copy
    $\MORPH R$ defined in Remark
    \ref{repr:state_space:rem} whose state-space is 
    $\mathbb{R}^{n}$. 
    In the rest of the construction,  we
    assume that $\mathcal{X}=\mathbb{R}^{n}$ for $n=\dim \mathcal{X}$ holds  and that $A_{q}$, $q \in Q$ are $n \times n$ matrices, and
    $C$ is a $p|Q| \times n$ matrix.
    Define the \emph{\BLSS $\Sigma_{R}$ associated with R}
    as follows. Let $\Sigma_{R}$ be of the form 
    (\ref{lin_switch0})  such that
    \begin{itemize}
    
   \item
    for $q \in Q$, the matrix $A_{q}$ of $\Sigma_{R}$ is identical to
    the state-transition matrix $A_{q}$ of $R$.
   \item
    For each $q \in Q$, the matrix $C_{q}$ is
    formed
    by the rows $(q-1)p+1, (q-1)p+2, \ldots, qp$ of $C$, i.e.
   \[
      C=\begin{bmatrix} C_{1}^{T}, & C_{2}^{T},  &  \cdots, & C_{\QNUM}^{T} \end{bmatrix}^{T}. 
     \]
    \item
    For each $q \in Q$, 
    $B_q=\begin{bmatrix} B_{(q,1)}, & \cdots & B_{(q,m)} \end{bmatrix}$.
    The initial state $x_0$ of $\Sigma_R$ is defined as
    $x_0=B_{0}$.
    \end{itemize}
  \end{Definition}  
   
    The intuition behind the definition of $\Sigma_R$ is the following.
    We would like $\Sigma_{R}$ to be such that if
    we apply Definition \ref{sect:real:const1} to it, then
    the resulting representation $R_{\Sigma_{R}}$ 
    should be close to $R$. 

The relationship between the various concepts introduced above
is as follows.
\begin{Theorem}
\label{sect:real:theo1}
     \begin{enumerate}
     \item
     \label{Part0pf}
      The Hankel-matrix $H_{\Psi_f}$ equals the Hankel-matrix $H_f$ of $f$.
     \item
     \label{Part01pf}
      The representations $R$ and $R_{\Sigma_R}$ are
      isomorphic, and $\Sigma_{R_{\Sigma}}=\Sigma$.
     \item
     \label{Part1pf}
      The \BLSS $\Sigma$ is a realization of the input-output map 
     $f$ if and only if 
     the associated representation
     $R_{\Sigma}$
    is a representation of $\Psi_{f}$.
    \item
   \label{Part2pf}
    The representation $R$
    is a representation of $\Psi_{f}$ if and only if 
    the associated \BLSS
    $\Sigma_{R}$ is a realization  of $f$.
   \item
   \label{Part3pf}
     The \BLSS $\Sigma$ is a minimal realization of the input-output map
     $f$ if and only if 
     the associated representation
     $R_{\Sigma}$ is a minimal representation of
     $\Psi_{f}$. 
   \item
   \label{Part4pf}
    The representation $R$ is a minimal
     representation of $\Psi_{f}$ if and only if 
     the associated \BLSS $\Sigma_{R}$ is a minimal realization of $f$.
   \item \label{Part5pf}
   The \BLSS $\Sigma$ is span-reachable (observable) if and only if
   the associated representation 
     $R_{\Sigma}$ is reachable (resp. observable).
   \item \label{Part6pf}
 
   The representation 
   $R$ is reachable (observable) if and only if 
   the associated \BLSS $\Sigma_{R}$ is \WR\ (resp. observable).

  \item
   \label{Part7pf}
  Assume that $\Sigma_1$ and $\Sigma_2$ are two \BSLSS
  with the state-spaces $\mathbb{R}^{n}$ and $\mathbb{R}^{n_a}$ respectively.
  A matrix $\MORPH \in \mathbb{R}^{n_a \times n}$ 
  is a \BLSS  morphism $\MORPH:\Sigma_1 \rightarrow \Sigma_2$ 
  if and only if
  $\MORPH:R_{\Sigma_1} \rightarrow R_{\Sigma_2}$ is a 
  representation morphism, if $\MORPH$ is interpreted as a linear map.
  \end{enumerate}
 \end{Theorem} 
 The statements of Theorem \ref{sect:real:theo1} above are summarized
 in Table \ref{tab:Summary}.

\begin{proof}[Proof of \ref{sect:real:theo1}]
\textbf{Proof of Part \ref{Part0pf}.} 
 Straightforward.

\textbf{Proof of Part \ref{Part01pf}.}
 Straightforward.

\textbf{Proof of Part \ref{Part1pf} and Part \ref{Part2pf}. }
 The proof is analogous to the proof of Theorem 10 from
 \cite{MP:BigArticlePartI}.
  First, note that if $R$ is a representation of $\Psi_{f}$, then $R$ satisfies the assumptions of Definition \ref{sect:real:const2}.  
Since $R$ is isomorphic to $R_{\Sigma_{R}}$, Part \ref{Part2pf} follows 
 from Part \ref{Part1pf}. Part \ref{Part1pf} follows by noticing that
 $\Sigma$ is a realization of $f$, if and only if
  for all $q_0 \in Q$, $j=1,\ldots,m$, $w \in Q^{*}$,
\begin{equation}
\label{sect:real:pow:eq3}
\begin{split}
& \mathbb{S}_{(q_0,j)}(w)=
    \begin{bmatrix} C_{1}^{T}, & C_{2}^{T}, & \cdots, & C_{\QNUM}^{T}
    \end{bmatrix}^{T}
    A_{w}B_{q_{0}}e_{j}  \mbox{ and } \\
    & \mathbb{S}_{0} (w)=
    \begin{bmatrix} C_{1}^{T}, & C_{2}^{T}, & \cdots, & C_{\QNUM}^{T}
    \end{bmatrix}^{T} 
    A_{w}x_0. 
 \end{split}
\end{equation}
  The above statement follows from Lemma \ref{sect:io:lemma1}, by 
  taking into account the definition of $\mathbb{S}_0$ and $\mathbb{S}_{(q_0,j)}$.
But \eqref{sect:real:pow:eq3} is equivalent to 
$R_{\Sigma}$ being a representation of $\Psi_f$.
Indeed, the matrix 
\( \begin{bmatrix}
     C_{1}^{T} & C_{2}^{T}, & \cdots & C_{\QNUM}^{T} \end{bmatrix}^{T}
   \)
in the right-hand side of \eqref{sect:real:pow:eq3} equals 
the readout matrix $C$ of $R_{\Sigma}$, and the vectors $B_{q_0}e_j$ and $x_0$ coincide with the initial states $B_{(q_0,j)}$ and $B_{0}$ of
$R_{\Sigma}$. Hence, \eqref{sect:real:pow:eq3} in fact says that
$\mathbb{S}_{j}(w)=CA_{w}B_j$ for all $w \in Q^{*}$,
$j \in J_f$, i.e. that $R_{\Sigma}$ is a representation of $\Psi_f$.

\textbf{Proof of Part \ref{Part3pf} and Part \ref{Part4pf}.}
Follows from Part \ref{Part1pf} and Part \ref{Part2pf}, by
 noticing that $\dim \Sigma=\dim R_{\Sigma}$ and
$\dim R=\dim \Sigma_R$.
\textbf{Proof of Part \ref{Part5pf} and \ref{Part6pf}. }
   Since $R_{\Sigma_R}$ is isomorphic to $R$,
   it is enough to prove Part \ref{Part5pf}.
   To that end it is enough to show that
   $W_{R_{\Sigma}}=\IM R(\Sigma)$ and $O_{R_{\Sigma}}=\ker O(\Sigma)$, i.e.
   the image of the reachability matrix of $\Sigma$ equals 
   the space $W_{R_{\Sigma}}$ of $R_{\Sigma}$, 
   and the kernel of the observability
   matrix of $\Sigma$ equals $O_{R_{\Sigma}}$.

   Assume that $R_{\Sigma}$ is of the form \eqref{repr:def0}, with
   $\mathcal{X}=\mathbb{R}^n$, $\POWO=p|Q|$ and $X=Q$.
   To see that $\IM R(\Sigma)=W_{R_{\Sigma}}$, 
   notice that $\IM \mathcal{R}(\Sigma)$ is the linear span
   of the columns of matrices $A_{w}B_{q}$ and vectors $A_{w}x_0$, 
   $q \in Q$, $w \in Q^{*}$, $|w| <  n$. But the initial states
   $B$
   of $R_{\Sigma}$ consists of the 
   columns of the matrices $B_q$, $q \in Q$, and of the vector $x_0$.
   Hence, $\IM \mathcal{R}(\Sigma)$ is spanned by vectors
   $A_{w}B_j$, $j \in J_f$ and hence it equals $W_{R_{\Sigma}}$.

  Similarly, the kernel of $O(\Sigma)$ equals the intersection of
  $\ker C_qA_{w}$, $q \in Q$, $w \in Q^{*}$, $|w| < n$.
   It is easy to see that
  $\bigcap_{q \in Q} \ker C_{q}A_{w}=\ker CA_{w}$, hence,
  $\ker O(\Sigma)$ is the intersection of all spaces $\ker CA_{w}$,
  $w \in Q^{*}$, $|w| < n$. But the latter intersection equals $O_{R_{\Sigma}}$.

%

\textbf{Proof of Part \ref{Part7pf}.}
    The proof is analogous to the proof of Lemma 10 of 
    \cite{MP:BigArticlePartI}.
     Since the state-spaces of $R_{\Sigma_1}$ and $\Sigma_1$ 
     are the same, and
     the state-spaces of $R_{\Sigma_2}$ and $\Sigma_2$ are the same, 
     $\MORPH$ can indeed be viewed both as a potential
     representation morphism from $R_{\Sigma_1}$ to $R_{\Sigma_2}$
     and as a potential \BLSS morphism from $\Sigma_1$ to $\Sigma_2$.
     Then it is enough to prove that
     $\MORPH$ satisfies \eqref{repr:morph:eq2} with $R=R_{\Sigma_1}$ and
     $\widetilde{R}=R_{\Sigma_2}$ if and only if $\MORPH$
     satisfies Definition \ref{sect:problem_form:lin:morphism}.
     The latter proof is routine.
      Indeed, 
    assume that  $\Sigma_1$ 
    is of the form (\ref{lin_switch0}) and 
    that $\Sigma_2$ is of the form
       \[ \Sigma_2=(n^{'}, Q, \{ (A_{q}^{'},B_{q}^{'},C_{q}^{'}) \mid q \in Q\},x_{0}^{'}).
	\]
       Assume that $R_{\Sigma_1}$ is of the form \eqref{repr:def0} and
     \( R_{\Sigma_2}=(\mathbb{R}^{n^{'}},
       \{ A_{q}^{'} \}_{q \in Q}, B^{'}, C^{'})
    \) where
    $B^{'}=\{B^{'}_{j} \mid j \in J_{f}\}$.
    Note that the matrices $A_q$ and $A_{q}^{'}$ of $R_{\Sigma_1}$, respectively  $R_{\Sigma_2}$,  coincide with the corresponding matrices 
of $\Sigma_1$ and $\Sigma_2$.
    Then
    $\MORPH$ is a \BLSS morphism if and only if  
   \begin{multline*}
		(\forall q \in Q: \MORPH A_{q}=A^{'}_{q}\MORPH, C_{q}=C_{q}^{'}\MORPH,
       \MORPH B_{q}=B_{q}^{'}) \\ \mbox{ and } \MORPH x_0=x_0^{'}.
   \end{multline*}
    But $\forall q \in Q: C_{q}=C_{q}^{'}\MORPH$ is equivalent to
    $C=C^{'}\MORPH$, since
    \[
     \begin{split}
       C&=\begin{bmatrix} (C_{1})^{T}, & \cdots, & (C_{\QNUM})^{T} \end{bmatrix}^{T} \\
      & =\begin{bmatrix} (C^{'}_{1}\MORPH)^{T}, & \cdots, & (C^{'}_{\QNUM}\MORPH)^{T} \end{bmatrix}^{T}=
       C^{'}\MORPH. 
     \end{split}
   \]
   Similarly, $\MORPH B_{q}=B^{'}_q$ is equivalent to:  
  $\forall l=1,\ldots,m, \MORPH B_q e_l=\MORPH B_{(q,l)}=B^{'}_qe_l=B_{(q,l)}^{'}$. This, together with $\MORPH x_0=x_0^{'}$, implies that
   $\MORPH B_{j}= B^{'}_j$ for all $j \in J_f$.

  Hence, we have established that $\MORPH$ is a \LSS\  morphism if and only if
  $\forall q \in Q: \MORPH A_{q}=A^{'}_{q}\MORPH$,
  $C=C^{'}\MORPH$, and $\forall j \in J_f: \MORPH B_j= B_j^{'}$. 
  But this means that $\MORPH:R_{\Sigma_1} \rightarrow R_{\Sigma_2}$ 
  is a representation morphism.
 \end{proof}
\begin{proof}[Proof Theorem \ref{theo:min}]
  By Theorem \ref{sect:real:theo1}, Part \ref{Part3pf}, 
   $\Sigma$ is a minimal \LSS\ 
  realization of $f$  
  if and only if $R=R_{\Sigma}$ is minimal.
  By Theorem \ref{sect:form:theo1}, $R$ is minimal if and only if  
  $R$ is reachable and observable.
  By Theorem \ref{sect:real:theo1}, Part \ref{Part5pf},
  the latter is equivalent to
  $\Sigma$ being \WR\ and observable.
 Next, we show that minimal \LSS\ realizations of $f$ are isomorphic.
 Let $\Sigma$ and $\hat{\Sigma}$ be two minimal \LSS\ realizations of $f$.
 By Theorem \ref{sect:real:theo1}, Part \ref{Part3pf},
 $R_{\Sigma}$ and $R_{\hat{\Sigma}}$ 
 are minimal representations of $\Psi_{f}$. Then from Theorem
 \ref{sect:form:theo1} it follows that there exists a
  isomorphism  
  $\MORPH:R_{\hat{\Sigma}} \rightarrow R_{\Sigma}$. From
  Part \ref{Part7pf} of Theorem \ref{sect:real:theo1} is
  then follows that
  $\MORPH: \hat{\Sigma} \rightarrow \Sigma$
  is an isomorphism.  Finally, the correctness of
  Procedure \ref{LSSmin} is shown in Remark \ref{LSSmin:pf}.
\end{proof}
  \begin{proof}[Proof of Theorem \ref{sect:real:theo2}]
  \textbf{Necessity} \\
   Assume that $\Sigma$ is a \BLSS which is a realization of $f$.
   Then by Lemma \ref{sect:io:lemma1}, $f$ has a \GCR. 
  Moreover, by Theorem \ref{sect:real:theo1},
   $R_{\Sigma}$ is a representation of $\Psi_{f}$, i.e. $\Psi_{f}$ is 
   rational.
   By Theorem \ref{sect:real:theo1}, Part \ref{Part0pf}, and Theorem \ref{sect:form:theo1}, 
   the latter implies that $\Rank H_{f} < +\infty$.

  \textbf{Sufficiency} \\
   Assume that $f$ has a \GCR\ and $\Rank H_{f} < +\infty$.
   Then by 
   Theorem \ref{sect:real:theo1}, Part \ref{Part0pf}, and Theorem \ref{sect:form:theo1}, 
    $\Psi_{f}$ is rational, i.e. it has a representation $R$. 
    Then 
    by Theorem \ref{sect:real:theo1} the \LSS\ $\Sigma_R$ is a realization
   of $f$, i.e. $f$ has a realization.
   
  Finally, the correctness of Procedure \ref{LSSfromHankel}
  follows from Remark \ref{sect:real:rem1} below.
\end{proof}   

Now we are ready to analyze Procedure \ref{LSSreach},\ref{LSSobs}, \ref{LSSmin} and \ref{LSSfromHankel}.
\begin{Remark}[Correctness of Procedure \ref{LSSreach}]
\label{LSSreach:pf}
 Procedure \ref{LSSreach} is equivalent to the following procedure.
 Apply Procedure
 \ref{repr:constr:reach} 
 to $R_{\Sigma}$ to obtain $R_r$. Then
 $\Sigma_r$ from Procedure \ref{LSSreach} and $\Sigma_{R_r}$
 are isomorphic. It then follows that
 $\Sigma_r$ is span-reachable, since $R_r$ is reachable, and 
 $\Sigma_r$ and $\Sigma$ have the same input-output map, since
 both $R_{\Sigma}$ and $R_r$ are representations of $\Psi_{y_{\Sigma}}$.
\end{Remark}
\begin{Remark}[Correctness of Procedure \ref{LSSobs}]
\label{LSSobs:pf}
 Procedure \ref{LSSobs} is equivalent to the following procedure.
 Apply
Procedure \ref{repr:constr:obs} to $R_{\Sigma}$ to obtain an
 observable representation $R_o$.
  It follows that $\Sigma_o$ from Procedure \ref{LSSobs} and 
  $\Sigma_{R_o}$ are isomorphic. Since $R_o$
 is observable, $\Sigma_o$ is observable as well.  If $\Sigma$ is
 span-reachable, then $R_{\Sigma}$ is reachable. Hence, then $R_o$ is
 reachable and thus $\Sigma_o$ is span-reachable. Finally, both
 $R_{\Sigma}$ and $R_o$ are representations of $\Psi_{y_{\Sigma}}$, 
 from which it follows that the input-output maps of $\Sigma$ and
 $\Sigma_o$ coincide.
\end{Remark}
\begin{Remark}[Correctness of Procedure \ref{LSSmin}]
 \label{LSSmin:pf}
 Procedure \ref{LSSmin} can be restated as follows. 
 Apply Procedure \ref{repr:constr:can} to $R_{\Sigma}$ and denote
 the resulting minimal representation by $R_m$. It then follows
 that $\Sigma_m$ from Procedure \ref{LSSmin} is
 isomorphic to $\Sigma_{R_m}$. Since
 by Theorem \ref{sect:real:theo1} $\Sigma_{R_m}$ is a minimal realization of
 $y_{\Sigma}$, then so is $\Sigma_m$.
\end{Remark}
 \begin{Remark}[Correctness of Procedure \ref{LSSfromHankel}]
 \label{sect:real:rem1}
 Procedure  \ref{LSSfromHankel} can be reformulated as follows.
  Use
  Procedure \ref{reprfromhankel}, to construct a 
  minimal representation $R$ of $\Psi_{f}$ from $H_{f}=H_{\Psi_{f}}$.
  Then by
  Theorem \ref{sect:real:theo1}, $\Sigma_{R}$ will be a 
  minimal realization of $f$. It is easy to see that
  the \BLSS $\Sigma_f$ from Procedure \ref{LSSfromHankel}
  is isomorphic to $\Sigma_R$.
 \end{Remark}
We will continue with the proof of Theorem \ref{part_real_lin:theo1}.
\begin{proof}[Proof of Theorem \ref{part_real_lin:theo1}]
 The proof is almost the same as that of
 the continuous-time case, described in \cite{MP:PartReal}.
 From Theorem \ref{sect:real:theo1} it follows that 
 $H_{f,K,L}$ coincides with
 $H_{\Psi_f,K,L}$, and
hence, $\Rank H_{f,N,N}=\Rank H_{f}$ is equivalent to
 $\Rank H_{\Psi_f,N,N}=\Rank H_{\Psi_f}$.

Assume now that $\Rank H_{f,N,N}=\Rank H_{f}$.
Then
the representation $R_{N}$ from Theorem \ref{part_real_pow:theo} 
is well-defined and it is a minimal
representation of $\Psi_{f}$.
Consider Algorithm \ref{alg0} and the decomposition
defined there. 
Then $\IM H_{f,N,N+1}=\IM \mathbf{O}$ and
there exists a left inverse
$\mathbf{O}^{+} \in \mathbb{R}^{n \times I_N}$ of $\mathbf{O}$ 
such that $\mathbf{O}^{+}\mathbf{O}=I_n$.

 Consider the linear map
$\MORPH:\IM H_{f,N,N+1} \rightarrow \mathbb{R}^n$, where
$\MORPH(x)=\mathbf{O}^{+}x$ for all $x \in \IM H_{f,N,N+1}$
and recall that $H_{f,N,N+1}=H_{\Psi_f,N,N+1}$.
It then follows
that $\MORPH$ is a linear isomorphism, and its inverse is $\mathbf{O}$.
Moreover, the isomorphic copy 
\[ \MORPH R_N = (\mathbb{R}^{n}, \{ \MORPH A_{q} \MORPH^{-1}\}_{q \in Q}, \{ \MORPH(B_j) \mid j \in J_{f}\}, C\MORPH^{-1}) \]
of $R_N$ is also a minimal representation 
of $\Psi_{f}$.

Consider now the \LSS\  
$\Sigma_{\MORPH R_{N}}$ associated with $\MORPH R_{N}$.
It is easy to see that the \LSS\ 
$\Sigma_{\MORPH R_{N}}$ 
satisfies (\ref{alg0:eq-2}-\ref{alg0:eq1}) and hence
it coincides with the \LSS\ $\Sigma_N$ returned by Algorithm \ref{alg0}. 
Theorem \ref{sect:real:theo1}
it follows then that $\Sigma_N$ is a minimal realization of $f$.

Assume that there exists an \LSS\ realization $\Sigma$ of $f$, such that $\dim \Sigma \le N+1$.  Then by Theorem \ref{sect:real:theo2}, $\Rank H_{f}=\Rank H_{\Psi_f} \le \dim \Sigma \le N+1$.  Hence, by Theorem \ref{part_real_pow:theo}, $\Rank H_{\Psi_{f}}=\Rank H_{\Psi_{f},N,N}$.
\end{proof}
We conclude this section with the following remark.
\begin{Remark}[Continuous-time case]
 If instead of a discrete-time system we consider a 
 continuous-time system $\Sigma$, 
 then the constructions of $R_{\Sigma}$ and $\Sigma_R$
 are exactly the same. The construction of $\Psi_f$ differs only in
 the way the Markov-parameters $S_{j}^f(q_0vq)$ and
 $S_{0}^f(vq)$, $v\in Q^*$, $q,q_0\in Q$, $j=1,\ldots,m$, are derived from the input-output map $f$. However,
 $S_{0}^f(vq)=C_qA_vx_0$ and $S_{j}^f(q_0vq)=C_qA_vB_{q_0}e_j$
 also holds for the continuous-time case, if $\Sigma$ is a realization
 of  $f$. A detailed description of the continuous-time case
 can be found in \cite{MP:BigArticlePartI,MP:RealForm}.
\end{Remark}
\begin{table}[h!]
\centering
\begin{tabular}{|ccc|}
\hline
Realization of $f$ & \vline & Representation of $\Psi_f$ \\
\hline
$\Sigma=\Sigma_{R_\Sigma}$  & $\Longleftrightarrow$ & $R_\Sigma$\\
$\Sigma_R$  & $\Longleftrightarrow$ & $R=R_{\Sigma_R}$ \\
observable, span-reachable & $\Longleftrightarrow$ & observable, reachable\\
minimal & $\Longleftrightarrow$ & minimal\\
$\MORPH$, \LSS morphism & $\Longleftrightarrow$ & $\MORPH$, representation morphism\\
\hline
\end{tabular}
\caption{Correspondence between \BSLSS  and representations}
\label{tab:Summary}
\end{table}

\section{Conclusions}
  We presented realization theory for discrete-time linear
  switched systems. The results and the proof techniques 
  resemble the ones for continuous-time linear switched systems presented in our previous work. 

\appendix
\section{Technical proofs}
\label{appA}
\begin{proof}[Proof of Lemma \ref{sect:io:lemma1}]
 Consider the input-output map $y_{\Sigma}$ of $\Sigma$.
 By induction on $t$, it follows that if
 $w=(v,u) \in \HYBINP^{+}$, 
 $v=q_0\cdots q_t$, $u=u_0\cdots u_t$, $t \ge 0$,
 $q_0,\ldots,q_t \in Q$, $u_0,\ldots, u_t \in \mathbb{R}^m$,
 then 
 \begin{equation}
 \label{sect:io:lemma1:pf:eq1}
   \begin{split}
     & y_{\Sigma}(w)=C_{q_t}A_{v_{0|t-1}}x_0+  
    \sum_{j=0}^{t-1} C_{q_t}A_{v_{j+1|t-1}}B_{q_j}u_{j}.
   \end{split}
 \end{equation}
 Consider  the Markov-parameters $S_{0}^{y_{\Sigma}}(sq)$,
 $S^{y_{\Sigma}}_{j}(q_0sq)$, $q,q_0 \in Q$, $s \in Q^{*}$,
 $j=1,\ldots,m$, of $y_{\Sigma}$.
 It then follows from \eqref{sect:io:lemma1:pf:eq1} and the 
 definition of Markov-parameters that for all $s \in Q^{*}$,
 \begin{equation}
 \label{sect:io:lemma1:pf:eq2}
    S^{y_{\Sigma}}_{0}(sq)=C_qA_sx_0 \mbox{ and } 
   S^{y_{\Sigma}}_{j}(q_0sq)=C_qA_sB_{q_0}e_j.  
 \end{equation}
 Notice that \eqref{sect:io:lemma1:pf:eq1} -- \eqref{sect:io:lemma1:pf:eq2}
 implies that $y_{\Sigma}$ has a generalized convolution representation.

 Assume that $\Sigma$ is a realization of $f$. Then $y_{\Sigma}=f$. 
 Then from \eqref{sect:io:lemma1:pf:eq1}--\eqref{sect:io:lemma1:pf:eq2} it follows
 that $f$ has a generalized convolution representation and 
 \eqref{sect:io:lemma1:eq1} holds. 
 Conversely, assume that $f$ has a generalized convolution representation and
 that \eqref{sect:io:lemma1:eq1} holds. From \eqref{sect:io:lemma1:eq1} 
it follows that
 the Markov-parameters of $y_{\Sigma}$ and $f$ coincide, i.e.
 $S^{y_{\Sigma}}_0(sq)=S^f_0(sq)$ and $S^{y_{\Sigma}}_{j}(q_0 s q)=S^f_{j}(q_0 s q)$ for all $s \in Q^{*}$,
 $q,q_0 \in Q$, $j=1,\ldots,m$. Since both $y_{\Sigma}$ and $f$ admit a 
 generalized convolution representation, by Remark \ref{rem:Unicity-of-from-MP} 
they are equal. The latter means that $\Sigma$ is a realization of $f$.
\end{proof}
\begin{proof}[Proof of Theorem \ref{sect:real:lemma1}]
  It is enough to show that for any family of $n \times n$
  matrices $F_q$, $q \in Q$ and any matrix $G \in \mathbb{R}^{n \times l}$ for some $l > 0$ the following holds. Define the matrix
  \( \mathcal{R}_k=\begin{bmatrix} 
      F_{v_1}G & \ldots & F_{v_{M_{k+1}}}G \end{bmatrix}
  \) for $k \in \mathbb{N}$.
  That is, $\mathcal{R}_k$ is the span of the column vectors 
  of $F_vG$, $v \in Q^{<k+1}$.
  Here we applied Notation \ref{repr:not1} to $F_q, q \in Q$ to obtain
  the matrices $F_v, v \in Q^{*}$.
  Define the subspace $\mathcal{I}$ as the space
  spanned by the column vectors of the matrices $F_vG$, $v \in Q^{*}$.
  If we can show that $\IM \mathcal{R}_{n-1}=\mathcal{I}$, then the
  statement of the theorem follows easily. 

  Indeed, 
  it is easy to see that the linear span of all reachable states
  of $\Sigma$ equals $\mathcal{I}$, if we set $F_q=A_q$, $q \in Q$ and
  $G=\widetilde{B}$. Moreover, in this 
  case $\mathcal{R}_{n-1}=\mathcal{R}(\Sigma)$.
  Hence, 
  $\Rank \mathcal{R}(\Sigma)=n$ is equivalent to $\mathcal{I}=\mathbb{R}^{n}$, which in turn is equivalent to span-reachability of $\Sigma$.
  Similarly, if we set $F_q=A_q^T$ and $G=\widetilde{C}^T$, then
  $\mathcal{O}(\Sigma)^T=\mathcal{R}_{n-1}$ and 
  $\mathcal{I}$ is the orthogonal complement of
  $\bigcap_{v \in Q^{*}} \ker \widetilde{C}A_{v}$. 
  From \cite{Sun:Book} it follows that $\Sigma$ is observable if and only if 
  $\bigcap_{v \in Q^{*}} \ker \widetilde{C}A_{v}=\{0\}$, which is
  equivalent to $\IM \mathcal{R}_{n-1}=\mathcal{I}=\mathbb{R}^n$. 
  The latter is equivalent to $\Rank \mathcal{O}(\Sigma)=n$.

 We proceed to show $\mathcal{I}=\IM \mathcal{R}_n$. The proof
 is the same as the one of an analogous statement for
 rational representations or state-affine systems \cite{MP:Phd,Son:Real}. We repeat it for the sake of completeness.
 It is easy to see that $\IM \mathcal{R}_k \subseteq \mathcal{I}$ for all $k \in \mathbb{N}$ and
 $\IM \mathcal{R}_{k} \subseteq \IM \mathcal{R}_{k+1}$. By a 
 dimensionality argument it follows that
 there exist $0 \le k_{*} \le n-1$, such that
 $\IM \mathcal{R}_{k_{*}}=\IM \mathcal{R}_{k_{*}+1}$. From this, by noticing that 
 $\IM \mathcal{R}_{k+1}=\IM G + \sum_{q \in Q} \IM F_{q}\mathcal{R}_k$, it follows that
 $\mathcal{I}=\IM \mathcal{R}_{k_{*}}$. Since $\IM \mathcal{R}_{k_{*}} \subseteq \IM \mathcal{R}_{n-1}$, we then obtain that 
 $\IM \mathcal{R}_{n-1}=\mathcal{I}$.

\end{proof}
\end{document}